\documentclass[11pt,reqno]{amsart}
\usepackage[colorlinks,linkcolor=blue,citecolor=blue,urlcolor=magenta,bookmarks=false,hypertexnames=true]{hyperref}
\usepackage{tikz}
\usepackage{pdflscape}
\usepackage{tikz-3dplot}
\usepackage[margin=1in]{geometry}
\usepackage{accents}
\usepackage{amsmath}
\allowdisplaybreaks
\usepackage{amssymb}
\usepackage{comment}
\usepackage{epsfig}
\usepackage{graphicx}
\usepackage[capitalize]{cleveref}
\usepackage{enumitem,kantlipsum}
\usepackage{mathrsfs}
\newtheorem{theorem}{Theorem}[section]

\newtheorem{lemma}[theorem]{Lemma}
\newtheorem{corollary}[theorem]{Corollary}

\theoremstyle{definition}
\newtheorem{definition}[theorem]{Definition}

\theoremstyle{remark}

\numberwithin{equation}{section}
\usepackage[numbers]{natbib}
\usepackage{fancyhdr}
\usepackage{fancyhdr}
\renewcommand{\geq}{\geqslant}
\renewcommand{\leq}{\leqslant}
\usepackage[utf8]{inputenc} % Allows you to type special characters directly (e.g., é, ö, à)
\usepackage[T1]{fontenc}    % Proper font encoding for output. Better for copy-pasting from the PDF.
\usepackage{comment}
\usetikzlibrary{arrows.meta}
\usepackage{amsmath}   % The premier package for math typesetting
\usepackage{amssymb}   % Provides many extra mathematical symbols
\usepackage{bbm}
\newcommand{\Fpsq}{\mathbb{F}_{p^2}}
\newcommand{\Fq}{\mathbb{F}_{q}}

\usepackage{graphicx}  % For including images with \includegraphics
\usepackage{booktabs}  % For professional-looking tables (\toprule, \midrule, \bottomrule)
\usepackage{tikz-cd}
\usepackage{caption}   % For more control over figure/table captions
\theoremstyle{plain} % Default
\newtheorem{proposition}[theorem]{Proposition}
\newtheorem{cor}[theorem]{Corollary}

\theoremstyle{definition}

\newcommand{\N}{\mathbb{N}}

\newcommand{\indicator}{\mathbbm{1}}
\newcommand{\Ftwo}{\mathbb{F}_2}

\theoremstyle{remark}
\usepackage{hyperref}
\hypersetup{colorlinks=true,    linkcolor=blue,   filecolor=magenta,         urlcolor=cyan,   citecolor=red,}
\usepackage{lipsum}

\title{Clebsch-Gordan and the theta filtration for modular representations of $\mathrm{GL}_2({\mathbb F}_q)$}

\author{Srijeet Bhattacharjee,  Eknath Ghate, Shivansh Pandey, Sriram Veerapaneni
}

\address{Theoretical Statistics and Mathematics Unit, Indian Statistical Institute, 203 Barrackpore, Trunk Road, Kolkata 700108, West Bengal, India.}
\address{School of Mathematics, Tata Institute of Fundamental Research, Homi Bhabha Road, Mumbai 400005, India.}
\address{School of Mathematics, Tata Institute of Fundamental Research, Homi Bhabha Road, Mumbai 400005, India.}
\address{Department of Mathematics and Statistics, Indian Institute of Technology Kanpur, Kalyanpur, Kanpur 208016, Uttar Pradesh, India.}
\address{}
\email{srijeet.b2005@gmail.com}
\email{eghate@math.tifr.res.in}
\email{shivansh@math.tifr.res.in}
\email{sriramcv22@iitk.ac.in}
\date{\today}

\begin{document}

\maketitle % This command generates the title based on the \title, \author, and \date commands
%---------------Abstract---------------
\begin{abstract}
Let $p$ be a prime. We solve two problems in the mod $p$ representation theory of $\mathrm{GL}_2(\mathbb{F}_{q})$ where $q=p^f$.  We first prove a Clebsch-Gordan decomposition theorem for the tensor product of two  mod $p$ representations of  $\mathrm{GL}_2(\mathbb{F}_{q})$. As an application, we use this to guess the structure of quotients of %tenor products of 
symmetric power representations of $\mathrm{GL}_2(\mathbb{F}_{q})$ by submodules in the theta filtration. We then give a direct proof of this structure showing that such quotients are built out of
principal series representations. 
\end{abstract}

%---------------Introduction---------------
\section{Introduction}
\label{sec:intro}
The Clebsch-Gordan theorem is a fundamental result in representation theory over the complex numbers and has applications to quantum mechanics. It describes how the tensor product of two irreducible representations decomposes as a direct sum of irreducible sub-representations. For example, the Clebsch-Gordan theorem for the group $\mathrm{SU}_2(\mathbb C)$ says that the tensor product of two symmetric power representations of degrees $m \leq n$ of the standard representation is a direct sum over such symmetric power representations of degree varying between the difference $n-m$ and the sum $m+n$ of the original degrees. 

%Over the complex numbers, the representation theory is semi-%simple. This means that the tensor product of two finite-%dimensional representations can always be written as a %direct sum of irreducible representations. However, the %situation is quite different over a field of positive %characteristic. In that case, the representation theory is %not semi-simple, and the tensor product of two %representations does not always decompose into a direct sum %of irreducibles. 

Tensor products also arise naturally in modular (more precisely mod $p$, for $p$ a prime) representation theory.
%For intsance,
%while studying the structure and decomposition of symmetric %power representations. In particular, tensor product also %appears in the work of Ghate-Jana %\cite{GhateJana+2025+1503+1543} where they studied the %structure of certain quotients of symmetric power %representations. 
A Clebsch-Gordan theorem for mod $p$ representations of $\mathrm{GL}_2(\mathbb F_p)$ was proved by Glover \cite{glover}, where a complete description of the tensor product of two irreducible (symmetric power) representations is given. We note that the indecomposable constitutents in the decomposition may no longer be irreducible since as is well-known modular representation theory involves non semi-simple objects. The situation becomes even more difficult for the group $\mathrm{GL}_2(\mathbb F_q)$ for general $q = p^f$ for 
$f \geq 1$. Special cases of the Clebsch-Gordan theorem 
in this setting can be found in the literature (see, 
for instance, \cite{F.M.}). In this paper, we study the Clebsh-Gordan theorem for $\mathrm{GL}_2(\mathbb F_q)$ for $q=p^f$ in some degree 
of generality. We hope the results we obtain will be of independent interest.

As an application, we use our Clebsch-Gordan theorem
to study the structure of quotients of symmetric power representations by submodules in the theta filtration. Recall that the Dickson polynomial $\theta$ is given by $x^py-xy^p$ and has the property that $\mathrm{GL}_2({\mathbb F}_p)$ acts on it via the determinant character. Divisibility by various powers of this polynomial define a filtration - called the theta filtration - on the symmetric power representations of $\mathrm{GL}_2({\mathbb F}_p)$ which are modeled on homogeneous polynomials in the variables $x$ and $y$. It is well known that the sub-quotients in this filtration are principal series representations for
$\mathrm{GL}_2({\mathbb F}_p)$. In particular, the quotient of the symmetric power representation by the submodule
$\langle \theta^{m+1} \rangle$ generated by $\theta^{m+1}$ for $m \geq 0$ 
is built out of principal series representations. We wish to generalize
this result to the case of $\mathrm{GL}_2({\mathbb F}_q)$. 

%Each step in the filtration is generated by powers of Ghate–Jana polynomials, and hence it is natural to study their structure through the filtrations induced by these polynomials. Thus, we provide an alternative approach to study these quotients using the Ghate–Jana polynomials, also known as twisted Dickson polynomials.
In order to do this, we introduce some notation. 
For $r=(r_0,r_1,...,r_{f-1})$ a tuple of integers with $r_i \geq 0$ let $V_{r}$ or more simply just $(r_0,r_1,...,r_{f-1})$ denote the symmetric power representation of $\mathrm{GL}_2(\mathbb F_q)$ 
given by $$V_{r}:=\bigotimes_{i=0}^{f-1}(\mathrm{Sym}^{r_i} {\mathbb F}_q^2\circ \mathrm{Fr}^i)^{},$$ 
%\bar{\mathbb F}_p^2)^{\mathrm{Fr}^i},$$ 
where $\mathrm{Fr}^i$ denotes the $i$-th Frobenius twist. The $i$-th component in the tensor product above is modeled on homogeneous polynomials over ${\mathbb F}_q$ of degree $r_i$ in the variables $x_i$ and $y_i$.
In \cite{GhateJana+2025+1503+1543}, the authors introduce
the twisted Dickson polynomials (or Ghate-Jana polynomials) 
$$\theta_i=x_iy_{i-1}^p-y_ix_{i-1}^p$$ for $i \in \{0,1,...,f-1\}$ with the convention $-1=f-1$ 
and studied, for $m=(m_0,m_1,...,m_{f-1})$ a tuple of integers with  $0 \leq m_i \leq p-1$,  the quotient 
$$\frac{V_{r}}{\langle\theta_0^{m_0+1},\theta_1^{m_1+1},...,\theta_{f-1}^{m_{f-1}+1}\rangle}.$$ Such quotients are expected to appear naturally
when one is computing the reductions of Hilbert modular 
Galois representations. In any case, as in the case of $f=1$, one might ask whether this quotient representation is also built out of principal series representations.

In \cite{GhateJana+2025+1503+1543}, Ghate-Jana show  
that the above quotient is isomorphic to the tensor product of the principal series representation of  $\mathrm{GL}_2({\mathbb F}_q)$ obtained by inducing  
the character $d^{r-m}$ for $r-m:=\sum_{i=0}^{f-1}(r_i-m_i)p^i$ of the subgroup $B({\mathbb F}_q)$ of upper triangular matrices of $\mathrm{GL}_2({\mathbb F}_q)$, and the symmetric power representation $V_{m}$:
$$\mathrm{ind}^{\mathrm{GL}_2(\mathbb F_q)}_{B(\mathbb F_q)}d^{r-m} \otimes V_{m}.$$
 Since the Jordan-H\"older
factors of principal series representations are well known
\cite{Breuil}, \cite{diamond},
we may use our Clebsch-Gordan theorem to obtain information about the irreducible representations one obtains upon tensoring the principal series above with 
$V_{m}$. Packaging these irreducible representations 
allows us to guess the structure of the above quotient. For instance, in the first interesting test case of $f = 2$ and $m_0 = m_1=1$,  the above quotient  
indeed appears to be built out of four principal series representations.
If the $m_i$ are all equal to some common $m \geq 0$ for each $i$, let $V_{r}^{(m+1)}:=\langle \theta_0^{m+1},\theta_1^{m+1},...,\theta_{f-1}^{m+1} \rangle$.  (Following Glover, we also sometimes write $V_{r}^*$ for $V_{r}^{(1)}$ and $V_{r}^{**}$ for $V_{r}^{(2)}$.) By taking $m = \max \{m_i\}$, the quotient
above is a homomorphic image of the quotient $V_{r}/V_{r}^{(m+1)}$, so as far as studying Jordan-H\"older factors is concerned we may restrict our study to the special
latter case. We then spend the rest of the paper showing that $V_{r}/V_{r}^{(m+1)}$ for general $f \geq 1$ and $m \geq 0$ is built out of $(m+1)^f$ principal series representations. This is done by a direct method which studies certain sub-quotients in the theta filtration.

\section{Clebsch-Gordan Theorem}

The Clebsch-Gordan decomposition was studied by Glover \cite{glover} for the group $\mathrm{GL}_2(\mathbb F_q)$ for $q = p$. Here, we compute the tensor product of two symmetric power representations for an arbitrary finite field $\mathbb F_q$ for $q =p^f$ for $f \geq 1$.  
%As an application, at the end of this section, we compute the Jordan-H\"older factors of $V_r/V_r^{**}$ when $f = 2$ and identify four possible principal series 
%appearing in it.

\subsection{Clebsch-Gordan decomposition for GL\(_2(\Fpsq)\)} We begin by proving the  Clebsch-Gordan decomposition for GL\(_2(\Fpsq)\).
For non-negative integers $m_0,m_1$, let $(m_0,m_1)$ denote the collection of bihomogeneous polynomials over $\mathbb{F}_{p^2}$ in the variables $x_0,y_0,x_1,y_1$ of degree $m_0$ in $x_0,y_0$ and degree $m_1$ in $x_1,y_1.$ Then $(m_0,m_1)$ is a $\mathrm{GL}_2(\Fpsq)$-representation under the action
\[\alpha \cdot P_0(x_0,y_0)P_1(x_1,y_1) \mapsto P_0(ax_0+cy_0,bx_0+dy_0)P_1(a^px_1+c^py_1,b^px_1+d^py_1),\]
where $\alpha=\left(\begin{smallmatrix}
    a&b\\c&d
\end{smallmatrix}\right).$
It is not difficult to see that $(m_0,0) \otimes (0,m_1) \cong (m_0,m_1) $ via the map $P_0(x_0,y_0)\otimes P_1(x_1,y_1)\mapsto P_0(x_0,y_0) P_1(x_1,y_1).$ 

To derive our Clebsch-Gordan formula, we construct two exact sequences. The following folklore result can be found in 
%%is attributed to G. E. Wall, D. J. Glover and 
Kouwenhoven \cite[Proposition 1]{F.M.} and also appears as  \cite[Theorem 2.10]{doi:10.1142/S0219498826502701}. Here and below, we adopt the convention that $(m_0,m_{1})$ is the zero representation if any of the entries $m_i$ are negative.

\begin{lemma}\label{shaliniabhik}
    Let $m_0,n_0 \geq 0$. Then we have an exact sequence of $\mathrm{GL_2}(\Fpsq)$-representations:
    \[0 \to (m_0-1,0) \otimes(n_0-1,0)\otimes \det \xrightarrow{} (m_0,0) \otimes (n_0,0) \xrightarrow{} (m_0+n_0,0) \to 0.\]
    Moreover, this sequence splits 
    if $p \nmid \binom{n_0+m_0}{n_0}$. 
\end{lemma}
\noindent By tensoring the above  exact sequence with $(0,m_1),$ we obtain the exact sequence:
\[0 \to (m_0-1,m_1)\otimes(n_0 - 1,0) \otimes \det \to (m_0,m_1) \otimes (n_0,0) \to (m_0 +n_0,m_1)\to 0.\]
If $p \nmid \binom{n_0+m_0}{m_0}$, this sequence splits and we have 
\begin{equation}\label{identity1}
      (m_0,m_1) \otimes(n_0,0) \cong ((m_0-1,m_1)\otimes(n_0-1,0) \otimes \det) \oplus  (m_0+n_0,m_1).
\end{equation}
Similarly, we have:
\begin{lemma}\label{twistlemma1}
    Let $m_1,n_1 \geq 0$. Then we have an exact sequence of $\mathrm{GL}_2(\Fpsq)$-representations:
    \[0 \to (0,m_1-1) \otimes (0,n_1-1) \otimes \mathrm{det}^{p} \xrightarrow{} (0,m_1)\otimes(0,n_1) \xrightarrow{} (0,m_1 + n_1) \to 0. \]  
    Moreover, this sequence splits 
    if $p \nmid \binom{n_1+m_1}{n_1}$. 
\end{lemma}
    \begin{proof}
        Take Frobenius twist of the exact sequence in Lemma \ref{shaliniabhik}. For some formal properties of Frobenius see the beginning of Section \ref{generalf}.
    \end{proof}
\noindent By tensoring the above exact sequence 
with $(m_0,0),$ we get the exact sequence:
\[0 \to (m_0,m_1-1)\otimes(0,n_1-1) \otimes \mathrm{det}^p \to (m_0,m_1)\otimes (0,n_1) \to (m_0,m_1 + n_1) \to 0.  \]
If $p \nmid \binom{m_1 + n_1}{n_1}$, this sequence splits and we have 
    \begin{equation}\label{identity2}
    (m_0,m_1) \otimes(0,n_1) \cong ((m_0,m_1-1)\otimes(0,n_1-1) \otimes \mathrm{det}^p) \oplus  (m_0,m_1+n_1).   
  \end{equation}

% So if $p \nmid \binom{n_0 + m_0}{m_0}\binom{m_1 + n_1}{n_1}$ we have the following two isomorphisms:

%\begin{equation}\label{identity1}
%      (m_0,m_1) \otimes(n_0,0) \cong ((m_0-1,m_1)\otimes(n_0-1,0) \otimes \mathrm{det}) \oplus  (m_0+n_0,m_1),
%\end{equation}
%  \begin{equation}\label{identity2}
%    (m_0,m_1) \otimes(0,n_1) \cong ((m_0,m_1-1)\otimes(0,n_1-1) \otimes \mathrm{det}^p) \oplus  (m_0,m_1+n_1).   
%  \end{equation}

Combining \eqref{identity1}, \eqref{identity2}, we obtain
the following theorem:
\begin{theorem}
    Let $m_0,m_1,n_0,n_1 \geq 0$ be integers. If $p \nmid \binom{m_0 + n_0}{n_0}\binom{m_1 + n_1}{n_1}$, then
    \begin{eqnarray*}
    (m_0,m_1) \otimes (n_0,n_1) &\cong &((m_0-1,m_1-1)\otimes(n_0-1,n_1-1)\otimes \mathrm{det}^{p+1})\\ &&\quad \oplus~ ((m_0+n_0,m_1-1)\otimes(0,n_1-1) \otimes \mathrm{det}^p) \\
    &&\quad \oplus ~((m_0-1,m_1+n_1)\otimes(n_0-1,0)\otimes \mathrm{det} \\
    &&\quad \oplus~ (m_0 + n_0,m_1 + n_1).
    \end{eqnarray*}
\end{theorem}
 \begin{proof} We have
 % observe that $(m_0,m_1) \otimes (n_0,n_1)\cong(m_0,m_1) \otimes (n_0,0)\otimes (0,n_1)$  to obtain
   \begin{eqnarray*}
    (m_0,m_1) \otimes (n_0,n_1)&\cong&    (m_0,m_1) \otimes (n_0,0)\otimes (0,n_1) \\
    &\cong &[(m_0-1,m_1)\otimes(n_0-1,0) \otimes \mathrm{det} \oplus (m_0+n_0,m_1)]   \otimes (0,n_1) \> \text{ by } \eqref{identity1} \\
     &\cong & (m_0-1,m_1)\otimes (0,n_1)\otimes(n_0-1,0) \otimes \mathrm{det} \oplus (m_0+n_0,m_1)\otimes (0,n_1) \\
     &\cong& ((m_0-1,m_1-1)\otimes(0,n_1-1) \otimes \mathrm{det}^p) \otimes (n_0-1,0)\otimes \mathrm{det} \\
     &&\quad \oplus ~  (m_0-1,m_1+n_1)\otimes (n_0-1,0)\otimes \mathrm{det} \\
   &&\quad \oplus~ ((m_0+n_0,m_1-1)\otimes(0,n_1-1) \otimes \mathrm{det}^p) \oplus  (m_0+n_0,m_1+n_1)  \> \text{ by } \eqref{identity2}\\
    &\cong& (m_0-1,m_1-1)\otimes(n_0-1,n_1-1) \otimes \mathrm{det}^{p+1} \\
     && \quad \oplus ~ (m_0-1,m_1+n_1)\otimes (n_0-1,0)\otimes \mathrm{det}\\
   && \quad \oplus~ (m_0+n_0,m_1-1)\otimes(0,n_1-1) \otimes \mathrm{det}^p \oplus  (m_0+n_0,m_1+n_1). \quad \qedhere
   \end{eqnarray*}    
   \end{proof}

\noindent As a special case, we obtain the following result which we use later.

\begin{cor} \label{cgtheorem}
      Let $m_0,m_1\geq 0$ be integers. If $p \nmid (m_0+1)(m_1+1)$, then
    \begin{eqnarray*}
    (m_0,m_1) \otimes (1,1) &\cong &(m_0-1,m_1-1)\otimes\mathrm{det}^{p+1}\oplus (m_0+1,m_1-1)\otimes\mathrm{det}^p \\ 
    &&\quad \oplus~ (m_0-1,m_1+1)\otimes \mathrm{det} \oplus (m_0 + 1,m_1 +1).
    \end{eqnarray*} 
\end{cor}

  In a similar manner, we can also deduce the following generalization of \cite[(5.5) (a)]{glover}.
\begin{cor}\label{smallweight}
Let $0\le m_0\le n_0 \le p-1,$ $0\le m_1\le n_1 \le p-1$ be such that $m_0+n_0\le p-1$ and $m_1+n_1\le p-1$.
Then
\begin{eqnarray*}
(m_0,m_1) \otimes(n_0,n_1)    &\cong& \bigoplus_{i=0}^{m_0} \bigoplus_{j=0}^{m_1}~ (m_0+n_0-2i,m_1+n_1-2j) \otimes \mathrm{det}^{i+jp}.
\end{eqnarray*}
\end{cor}
\begin{proof}
Applying Lemma \ref{shaliniabhik} $m_0$ times, we obtain 
    \begin{eqnarray*}
        (m_0,m_1) \otimes(n_0,0) &\cong&((m_0-1,m_1)\otimes(n_0-1,0) \otimes \mathrm{det}) \oplus  (m_0+n_0,m_1)\\
&\cong&\bigoplus_{i=0}^{m_0}~(m_0+n_0-2i,m_1)\otimes \mathrm{det}^{i}. 
    \end{eqnarray*}
Similarly, by Lemma \ref{twistlemma1}, we have 
    \begin{eqnarray*}
        (m_0,m_1) \otimes(0,n_1) \cong \bigoplus_{j=0}^{m_1}~(m_0,m_1+n_1-2j)\otimes \mathrm{det}^{jp}. 
    \end{eqnarray*}
Using these two facts, we have    
\begin{eqnarray*}
        (m_0,m_1) \otimes(n_0,n_1) &\cong& (m_0,m_1) \otimes(n_0,0) \otimes (0,n_1) \\&\cong&\bigoplus_{i=0}^{m_0}~(m_0+n_0-2i,m_1)\otimes \mathrm{det}^{i}\otimes (0,n_1) \\
&\cong&\bigoplus_{i=0}^{m_0} \bigoplus_{j=0}^{m_1}~ (m_0+n_0-2i,m_1+n_1-2j)\otimes \mathrm{det}^{jp}\otimes \mathrm{det}^{i}\\
&\cong&\bigoplus_{i=0}^{m_0} \bigoplus_{j=0}^{m_1} ~(m_0+n_0-2i,m_1+n_1-2j) \otimes \mathrm{det}^{i+jp}. \qquad \qquad \qquad \qquad \quad \qedhere
    \end{eqnarray*} 
\end{proof}

The next lemma is an analogue of Glover \cite[(5.5) (b)]{glover}.

\begin{lemma}
\label{lem5}
     Let $0\le m_0\le n_0\le p-1$ be such that $p-2\le m_0+n_0\le 2p-2.$ We have 
     \begin{eqnarray*}
        (m_0,0)\otimes (n_0,0)&\cong&(p-m_0-2,0)\otimes (p-n_0-2,0)\otimes \mathrm{det}^{m_0+n_0+2-p}\\
        &&\quad \oplus~ (m_0+n_0+1-p,0)\otimes (p-1,0).
     \end{eqnarray*}
\end{lemma}
\begin{proof}
  Following \cite{glover}, we prove the lemma by induction on $m_0.$ Let's assume $m_0=0$, then $n_0$ can be $p-2$ or $p-1.$ In both the cases, we have a tautology. Now assume that $m_0=1.$ Then $n_0$ can be one of $p-3, p-2, p-1.$ Again for $n_0=p-3, p-1$ we have a tautology. The case $n_0=p-2$ follows from Lemma \ref{shaliniabhik}. Now we assume $1\le m_0'< p-1$ and that the statement is true for all $m_0\le m_0'$  and all possible values of $n_0.$ We have to prove the lemma for $m_0'+1$ and for all $n_0$ such that $p-2\le m_0'+1+n_0\le 2p-2.$ When $m_0'+1+n_0=p-2$, the statement is a tautology. If $m_0'+1+n_0=p-1$, then by Lemma \ref{shaliniabhik} we have 
    \begin{eqnarray*}
       (m_0'+1,0)\otimes (n_0,0)&\cong& (m_0',0)\otimes(n_0-1,0) \otimes \mathrm{det} \oplus  (m_0'+1+n_0,0)\\ 
       &\cong& (p-2-n_0,0)\otimes (p-2-(m_0'+1),0) \otimes \mathrm{det}\\
       &&\quad \oplus ~(0,0)\otimes (p-1,0)
    \end{eqnarray*}
    as desired. Hence we can assume that $p\le m_0'+1+n_0\le2p-2,$ i.e., $p-2\le m_0'-1+n_0 \leq 2p-4$. We compute $(1,0)\otimes (m_0',0)\otimes (n_0,0)$ in two ways using the associativity of the tensor product. By 
    Lemma~\ref{shaliniabhik}, we have
    \begin{eqnarray*}
        (1,0)\otimes (m_0',0)\otimes(n_0,0)&\cong& ((m_0'-1,0)\otimes \mathrm{det}\oplus (m_0'+1,0))\otimes (n_0,0)\\
        &\cong& (p-m_0'-1,0)\otimes (p-n_0-2,0)\otimes \mathrm{det}^{m_0'+n_0+2-p} \\
        &&\quad \oplus~ (m_0'+n_0-p,0)\otimes(p-1,0)\otimes \mathrm{det}\oplus (m_0'+1,0)\otimes (n_0,0)
    \end{eqnarray*}
    using the inductive hypothesis for $(m_0'-1,0)\otimes (n_0,0).$ Again, since $p-1\le m_0'+n_0$, by the inductive hypothesis we have 
    \begin{eqnarray*}
       (1,0)\otimes (m_0',0)\otimes(n_0,0)&\cong& (1,0)\otimes ((p-m_0'-2,0)\otimes (p-n_0-2,0)\otimes \mathrm{det}^{m_0'+n_0+2-p}\\
   &&\quad \oplus \;(m_0'+n_0+1-p,0)\otimes(p-1,0))\\
    &\cong&(p-m_0'-3,0)\otimes (p-n_0-2,0)\otimes \mathrm{det}^{m_0'+n_0+3-p}\\
    &&\quad \oplus \; (p-m_0'-1,0)\otimes(p-n_0-2,0) \otimes \mathrm{det}^{m_0'+n_0+2-p}\\
    &&\quad \oplus \; (m_0'+n_0-p,0)\otimes (p-1,0)\otimes \mathrm{det}\\
    &&\quad \oplus \; (m_0'+n_0+2-p,0)\otimes (p-1,0),
    \end{eqnarray*}
    where in the second step we have used Lemma \ref{shaliniabhik}.
  Comparing the above two decompositions of $(1,0)\otimes (m_0',0)\otimes(n_0,0)$ and using the Krull-Schmidt theorem on uniqueness of indecomposable factors up to order, we obtain 
  \begin{eqnarray*}
     (m_0'+1,0)\otimes (n_0,0)&\cong& (p-m_0'-3,0)\otimes (p-n_0-2,0)\otimes \mathrm{det}^{m_0'+n_0+3-p}\\
    &&\quad \oplus~(m_0'+n_0+2-p,0)\otimes (p-1,0)  
  \end{eqnarray*}
  which is the statement of the lemma for $m_0=m_0'+1.$ This completes the inductive step.
\end{proof}

 The next lemma can be proved by taking the Frobenius twist of Lemma \ref{lem5}.

\begin{lemma}\label{lem6}
     Let $0\le m_1\le n_1\le p-1$ be such that $p-2\le m_1+n_1\le 2p-2.$ We have 
     \begin{eqnarray*}
        (0,m_1)\otimes (0,n_1)&\cong&(0,p-m_1-2)\otimes (0,p-n_1-2)\otimes \mathrm{det}^{p(m_1+n_1+2-p)}\\
        &&\quad \oplus \; (0,m_1+n_1+1-p)\otimes (0,p-1).
     \end{eqnarray*}
\end{lemma}

Tensoring the two lemmas above, we obtain:
\begin{cor}\label{largeweights} Let $0\le m_0\le n_0\le p-1$ and  $0\le m_1\le n_1\le p-1$ be such that $p-2\le m_0+n_0\le 2p-2$ and $p-2\le m_1+n_1\le 2p-2.$ We have
     \begin{eqnarray*}
        (m_0,m_1)\otimes (n_0,n_1)&\cong&  (p-1,p-1)\otimes(m_0+n_0+1-p,m_1+n_1+1-p)\\
        &&\quad \oplus \;(p-m_0-2,p-m_1-2)\otimes (p-n_0-2,p-n_1-2)\\ 
        && \quad \;\;\;\;\otimes \; \mathrm{det}^{(m_0+n_0+2-p) + p(m_1 + n_1 + 2 - p)}\\
        &&\quad  \oplus \;  (p-m_0-2,0) \otimes \; (p-n_0-2,0) \\
        && \quad  \;\;\;\;\otimes \; (0,m_1+n_1+1-p) \otimes(0,p-1) \otimes \; \mathrm{det}^{m_0+n_0+2-p} \\
        &&\quad \oplus \; (m_0+n_0+1-p,0) \otimes(p-1,0) \\
        &&\quad \;\;\;\;\otimes \;(0,p-m_1-2) \otimes (0,p-n_1-2)   \otimes \mathrm{det}^{p(m_1+n_1+2-p)}.
     \end{eqnarray*}
\end{cor}

The first term on the right in the previous corollary can be rewritten using the following general fact which treats the case of tensor product 
with the projective module $(p-1,p-1)$ (note that neither
$(p-1,0)$ nor $(0,p-1)$ are projective since $p^2$ must divide the dimension of a projective $\mathrm{GL}_2({\mathbb F}_{p^2})$-module). This fact generalizes Glover \cite[(5.3)]{glover}. 

\begin{theorem}
    \label{projective}
    Let $(m_0,m_1)$ be a representation of $\mathrm{GL}_2(\mathbb F_{p^2}).$ Then we have
    \begin{eqnarray*}(m_0,m_{1}) \otimes (p-1,p-1) \cong ((m_{1}+1)p-1,(m_{0}+1)p-1).\end{eqnarray*}
\end{theorem}
\begin{proof} Again, the proof follows \cite{glover}. We define a map $\beta:(m_0,m_1)\rightarrow (m_1p,m_{0}p)$  by 
$$P_0(x_0,y_0)\cdot P_{1}(x_1,y_1)\mapsto P_{0}(x_{1}^{p},y_{1}^{p})\cdot P_{1}(x_{0}^{p},y_{0}^{p}).$$
Clearly, $\beta$ is an injective linear map. We check that it is $\mathrm{GL}_2(\mathbb F_{p^2})$-equivariant. 
For $\alpha = \left( \begin{smallmatrix} a & b \\ c & d \end{smallmatrix} \right) \in \mathrm{GL}_2(\mathbb F_{p^2})$, we have
\begin{eqnarray*}
 \beta(\alpha \cdot P_0( x_0,y_0)P_1( x_1,y_1) ) &=& \beta_{}(P_0(a^{} x_0 + c^{}y_0,b^{} x_0 + d^{}y_0)P_1(a^{p} x_1 + c^{p}y_1,b^{p} x_1 +d^{p}y_1))\\
 &=&P_0(a^{} x^{p}_{1} + c^{}y^{p}_{1},b^{} x^{p}_{1} + d^{}y^{p}_{1}) P_1(a^{p} x^{p}_{0} + c^{p}y^{p}_{0},b^{p} x^{p}_{0} + d^{p}y^{p}_{0})
 \\
&=&P_0(a^{p^2} x^{p}_{1} + c^{p^2}y^{p}_{1},b^{p^2} x^{p}_{1} + d^{p^2}y^{p}_{1}) P_1(a^{p} x^{p}_{0} + c^{p}y^{p}_{0},b^{p} x^{p}_{0} + d^{p}y^{p}_{0})
 \\
 &=&P_0((a^{p} x^{}_{1} + c^{p}y^{}_{1})^p,(b^{p} x^{}_{1} + d^{p}y^{}_{1})^p) P_1((a^{} x^{}_{0} + c^{}y^{}_{0})^p,(b^{} x^{}_{0} + d^{}y^{}_{0})^p)
 \\
 &=& \alpha \cdot P_{0}(x_{1}^{p},y_{1}^{p}) P_{1}(x_{0}^{p},y_{0}^{p})\\
 &=& \alpha\cdot\beta_{}(P_0(x_0,y_0)P_1(x_1,y_1)).
\end{eqnarray*}
Consider the composition of maps:
\begin{align*}
   (m_0,m_1) \otimes (p-1,p-1) \xrightarrow{\beta_{}\otimes\mathrm{Id}} (m_1p,m_{0}p)\otimes (p-1,p-1) \xrightarrow{\varphi} ((m_{1}+1)p-1,(m_{0}+1)p-1),
\end{align*} 
where the second map $\varphi$ is given by $P\otimes Q\mapsto PQ.$ Since both the maps are $\mathrm{GL}_2(\mathbb F_{p^2})$-equivariant, the composition is also $\mathrm{GL}_2(\mathbb F_{p^2})$-equivariant. The image of any monomial $x_{0}^{l_0}y_{0}^{m_0-l_0}x_{1}^{l_1}y_{1}^{m_1-l_1}\otimes x_{0}^{s_0}y_{0}^{p-1-s_0}x_{1}^{s_1}y_{1}^{p-1-s_1}$ under the composition of the two maps is given by $$x_{0}^{l_{1}p+s_{0}}y_{0}^{(m_{1}-l_1)p+p-1-s_{0}}  x_{1}^{l_{0}p+s_{1}}y_{1}^{(m_{0}-l_0)p+p-1-s_{1}}  .$$
As $l_0, l_1$ vary from $0$ to $m_0, m_1$ respectively, and as $s_{0}, s_1$ vary from from $0$ to $p-1,$ we get that $l_{0}p+s_{1}$ varies from $0$ to $(m_{0}+1)p-1$ and similarly $l_{1}p+s_{0}$ varies from $0$ to $(m_1+1)p-1.$ Thus the composition of the two maps is surjective. Now a comparison of the dimension of the two spaces shows that composition is an isomorphism.
\end{proof}

So far we have considered the two extreme cases $m_0+n_0,m_1+n_1 \leq p-1$ and $ p - 2\leq m_0 + n_0,m_1 + n_1 \leq 2p - 2$. The other two `cross' cases follow similarly by tensoring \eqref{identity1} with Lemma~\ref{lem6}
and \eqref{identity2} with Lemma~\ref{lem5}. We obtain:

\begin{cor}
    Let $0 \leq m_0 \leq n_0 \leq p-1$ and $0 \leq m_1 \leq n_1 \leq p-1$ be such that $m_0 + n_0 \leq p-1$ and $p-2 \leq m_1 + n_1 \leq 2p -2$. We have
 \begin{eqnarray*}
        (m_0,m_1)\otimes (n_0,n_1)&\cong&  (m_0 - 1, p - m_1 - 2) \otimes (n_0-1,p-n_1 - 2) \otimes \mathrm{det}^{p(m_1 + n_1 + 2 - p) + 1} \\
        &&\quad \oplus \; (m_0 + n_0, p - m_1 - 2) \otimes (0,p - n_1 - 2) \otimes \mathrm{det}^{p(m_1 + n_1 + 2 - p)}\\
        && \quad \oplus \; (m_0 - 1,m_1 + n_1 + 1 - p) \otimes (n_0 - 1,p-1) \otimes \mathrm{det} \\
        && \quad \oplus \; (m_0 + n_0, m_1 + n_1 + 1 -p) \otimes (0,p-1).
        \end{eqnarray*}
\end{cor}

\begin{cor}
    Let $0 \leq m_0 \leq n_0 \leq p-1$ and $0 \leq m_1 \leq n_1 \leq p-1$ be such that $p-2 \leq m_0 + n_0 \leq 2p -2$  and $0 \leq m_1 + n_1 \leq p-1$.  We have
 \begin{eqnarray*}
        (m_0,m_1)\otimes (n_0,n_1)&\cong&  (p - m_0 - 2,m_1 - 1) \otimes (p-n_0 - 2,n_1 - 1) \otimes \mathrm{det}^{(m_0 + n_0 + 2 - p) + p} \\
        && \quad \oplus \; (p - m_0 - 2,m_1 + n_1) \otimes (p - n_0 - 2,0) \otimes \mathrm{det}^{m_0 + n_0 + 2 - p}\\
        && \quad \oplus \; (m_0 + n_0 + 1 - p,m_1 -1) \otimes (p-1,n_1 - 1) \otimes \mathrm{det}^p \\
        && \quad \oplus \; (m_0 + n_0 + 1 -p,m_1 + n_1) \otimes (p-1,0).
        \end{eqnarray*}
\end{cor}

\subsection{Clebsch-Gordan decomposition for GL\(_2(\Fq)\)}\label{generalf} Now let $q = p^f$ for general $f \geq 1$. For non-negative integers $m_0, m_1, \ldots, m_{f-1}$, let $(m_0,m_1,\ldots,m_{f-1})$ denote the space of multi-homogeneous polynomials over $\Fq$ in the variables $z_0 = (x_0,y_0),z_1 = (x_1,y_1),\ldots, z_{f-1} = (x_{f-1},y_{f-1})$ of multi-degree $(m_0,m_1, \ldots,m_{f-1})$. Then $(m_0,m_1,\ldots,m_{f-1})$ is a $\mathrm{GL}_2(\Fq)$ representation under the action given by:

$$\alpha\cdot \prod_{i=0}^{f-1} P_i(x_i,y_i) = \prod_{i=0}^{f-1} P_i(a^{p^i}x_i+c^{p^i}y_i,b^{p^i}x_i+d^{p^i}y_i).$$

The map $\mathrm{Fr}$ takes $\left( \begin{smallmatrix}a & b \\ c & d\end{smallmatrix} \right)$ to $\left( \begin{smallmatrix}a^p & b^p \\ c^p & d^p\end{smallmatrix} \right)$. In general for a $\mathrm{GL_2}(\Fq)$-representation $V$, let $V \circ \mathrm{Fr}^i$ denote the representation obtained by twisting the action by the $i$-th power of Frobenius. For example, $
(m_0,0,\ldots,0) \circ \mathrm{Fr}^i \cong (0,\ldots,\underset{\text{$i$-th entry}}{m_0},\ldots,0)
$ and if $\mathrm{det}$ is the trivial representation twisted by the determinant, then $\mathrm{det}\circ \mathrm{Fr}^i = \mathrm{det}^{p^i}$ because $(a^{p^i}d^{p^i}-b^{p^i}c^{p^i})=(ad-bc)^{p^i}$ mod $p$. We note that $\mathrm{Fr}^i$ distributes over tensor products and direct sums, i.e., $(V\otimes W)\circ \mathrm{Fr}^i\cong (V\circ \mathrm{Fr}^i)\otimes( W\circ \mathrm{Fr}^i)$  and $(V\oplus W)\circ \mathrm{Fr}^i\cong (V\circ \mathrm{Fr}^i)\oplus( W\circ \mathrm{Fr}^i).$ If $f:V \to W$ is a $\mathrm{GL}_2(\mathbb F_q)$-equivariant map, then it is also an equivariant map from $V\circ \mathrm{Fr}^i $ to $W\circ \mathrm{Fr}^i $. To see this, note
\begin{align*}
   f\left(
   \left( \left( \begin{smallmatrix}
        a & b\\
        c & d
   \end{smallmatrix} \right) \circ \mathrm{Fr}^i \right)
   \cdot v
   \right)
   =  f\left( \left( \begin{smallmatrix}
        a^{p^i} & b^{p^i}\\
        c^{p^i} & d^{p^i}  \end{smallmatrix} \right) \cdot v\right) 
   = \left( \begin{smallmatrix}
        a^{p^i} & b^{p^i}\\
        c^{p^i} & d^{p^i}  \end{smallmatrix} \right) \cdot f(v)
   = \left( \left( \begin{smallmatrix}
        a & b\\
        c & d
   \end{smallmatrix} \right) \circ \mathrm{Fr}^i \right) \cdot f(v).
\end{align*}

As before, we start with the following folklore result from \cite{F.M.}.
% due to G. E. Wall, D. J. Glover and F. M. Kouwenhoven appears as theorem 2.10 in \cite{doi:10.1142/S0219498826502701}.
We adopt the convention that $(m_0,\cdots,m_{f-1})$ is the zero representation if any of the entries $m_i$ are negative.

\begin{lemma}
    Let $m_0,n_0\geq 0$. Then we have an exact sequence of $\mathrm{GL}_2(\Fq)$-representations
    \begin{align*}
    0 &\to (m_0-1 ,0,...,0)\otimes(n_0-1,0,...,0) \otimes \mathrm{det}^{} \\
    &\to (m_0,0,...,0)\otimes(n_0,0,...,0) \to (m_0+n_0,0,...,0) \to 0
    \end{align*}
which splits if $p \nmid \binom{n_0 + m_0}{n_0}.$
\end{lemma}
Take the Frobenius twist of the above sequence and observe that exactness is preserved. We get the following result:

\begin{lemma}\label{singletupleidentity}
    Let $0 \leq i \leq f-1 $. Let $m_i,n_i \geq 0$. Then we have an exact sequence of $\mathrm{GL}_2(\Fq)$-representations
    \begin{align*}
    0 &\to (0,...,0,m_i-1,0,... ,0)\otimes(0,...,0,n_i-1,0,... ,0) \otimes \mathrm{det}^{p^i} \\
    &\to (0,...,0 ,m_i,0,...,0)\otimes(0,...,0 ,n_i,0,...,0) \to (0,...,0,m_i + n_i,0,...,0) \to 0
    \end{align*}
which splits if $p \nmid \binom{n_i + m_i}{m_i}.$ 
\end{lemma}

Let $m = (m_0,\cdots,m_{f-1})$, let $e_i$ denote the vector with $1$ in its $i$-th coordinate and $0$ everywhere else (starting with $i = 0$). Just as in the $f = 2$ case, for every $ 0 \leq i \leq f-1$ such that $p \nmid \binom{n_i + m_i}{m_i}$ and $m_i,n_i \geq 1$, we have the isomorphism:
\begin{equation}\label{lemma5identity}
       m\otimes n_ie_i \cong ((m-e_i)\otimes (n_i-1)e_i\otimes \mathrm{det}^{p^i} )\oplus (m + n_ie_i).
\end{equation}

For \(l \in \{0,1\}^f \), define \(\psi_p(l):=\sum_{j=0}^{f-1}l_jp^j \in \mathbb{Z}.\)
For $m, n \in \mathbb{Z}^f$, denote \(m\odot n\) to be the tuple in \(\mathbb{Z}^f\) obtained by taking component-wise product of \(m\) and \(n\).
Let \(\indicator_f\) denote the element in \(\mathbb{Z}^f\) with all components as \(1\).

\begin{theorem} Let \(m=(m_0,m_1,...,m_{f-1}),\ n=(n_0,n_1,...,n_{f-1})\). If \(p\nmid \binom{m_i+n_i}{n_i}\) for all \(0\le i \le f-1\), then
    \[
    m\otimes n \cong \bigoplus_{l\in \{0,1\}^f}(m-l+(\indicator_f -l)\odot n )\otimes (l\odot n - l) \otimes \mathrm{det}^{\psi_p(l)}.
    \]
\end{theorem}
\begin{proof}
    Let $j$ be the largest index for which $n_j$ is non-zero. %Let $e_j$ denote the vector with $1$ in its $j^\mathrm{th}$ coordinate and $0$ everywhere else. 
    Let $W_j \subset \{0,1\}^f$ denote the subset consisting of tuples of the form $(i_0,\ldots,i_j,0,\ldots,0)$. We show by induction on $j$ the following statement:
    \[
    m\otimes n \cong \bigoplus_{l\in W_j}(m-l+(\indicator_f - l)\odot n ) \otimes (l\odot n - l) \otimes \mathrm{det}^{\psi_p(l)}.
    \]
   It is enough to prove the above statement for all $j$ between $0$ and $f-1$. The case $j = 0$ is the statement given by 
   \begin{align*}
    m\otimes n_0e_0\cong ((m-e_0)\otimes (n_0-1)e_0\otimes \mathrm{det} )\oplus (m + n_0e_0),
\end{align*}
   which is \eqref{lemma5identity} for $i=0$. Let's assume the result for $j = t < f-1$. Then for $j = t + 1$, we can use the isomorphism $n=\sum_{k=0}^{t+1}n_ke_k \cong (\sum_{k=0}^tn_ke_k) \otimes n_{t+1}e_{t+1}$ and the induction hypothesis to get: 
{\small   
\begin{align*}
    m\otimes n &\cong \left(\bigoplus_{l \in W_t}\bigg(m-l+(\indicator_f - l)\odot (\sum_{k=0}^tn_ke_k) \bigg)\otimes \bigg(l\odot (\sum_{k=0}^t n_ke_k)-l\bigg)\otimes \mathrm{det}^{\psi_p(l)} \right)\otimes n_{t+1}e_{t+1}\\
    &\cong \left(\bigoplus_{l \in W_t} \bigg(m - l+(\indicator_f - l)\odot (\sum_{k=0}^tn_ke_k)\bigg) \otimes n_{t+1}e_{t+1}\otimes \bigg(l\odot (\sum_{k=0}^tn_ke_k)-l\bigg) \otimes
    \mathrm{det}^{\psi_p(l)} \right)\\
    &\cong \bigoplus_{l\in W_t} \Bigg[ \left( \bigg(m - l - e_{t+1} + (\mathbbm{1}_f - l)\odot (\sum_{k=0}^tn_ke_k)  \bigg)\otimes (n_{t+1}-1)e_{t+1} \otimes (l\odot (\sum_{k=0}^tn_ke_k)-l)  \otimes \mathrm{det}^{p^{t+1}}\right)\\
    &  \;\;\;\;\;\;\;\;\;\;\;\; \oplus \left( \bigg(m - l + n_{t+1}e_{t+1} + (\mathbbm{1}_f - l)\odot (\sum_{k=0}^tn_ke_k)\bigg)\otimes (l\odot (\sum_{k=0}^tn_ke_k)-l) \right) \Bigg]\otimes \mathrm{det}^{\psi_p(l)} \\
    &\cong \bigoplus_{l\in W_t} \Bigg[ \Bigg( \bigg(m - l - e_{t+1} + (\mathbbm{1}_f - l - e_{t+1})\odot (\sum_{k=0}^{t+1}n_ke_k)  \bigg)\otimes \bigg((l+e_{t+1})\odot (\sum_{k=0}^{t+1}n_ke_k)-l-e_{t+1}\bigg) \\  & \;\;\;\;\;\;\;\;\;\;\;\otimes \mathrm{det}^{p^{t+1}}\Bigg)
    \oplus \left( \bigg(m - l + (\mathbbm{1}_f - l)\odot (\sum_{k=0}^{t+1}n_ke_k)\bigg)\otimes (l\odot (\sum_{k=0}^{t+1}n_ke_k)-l) \right) \Bigg]\otimes \mathrm{det}^{\psi_p(l)} \\
   &\cong  \bigoplus_{l \in W_{t+1}} (m-l+(\indicator_f - l)\odot n) \otimes (l \odot n - l) \otimes \mathrm{det}^{\psi_p(l)}.
\end{align*}}
The third isomorphism uses \eqref{lemma5identity}
with $i = t+1$. The fourth isomorphism uses the identity $$(n_{t+1}-1)e_{t+1} \otimes (l\odot (\sum_{k=0}^tn_ke_k)-l)=(l+e_{t+1})\odot (\sum_{k=0}^{t+1}n_ke_k)-l-e_{t+1}).$$
The last isomorphism follows from the fact that the two terms in the direct sum correspond to the terms with $l_{t+1} = 1,0$ respectively. This completes the induction and proves the theorem.
\end{proof}

The following corollary generalizes Corollary~\ref{smallweight}.
\begin{corollary}
  Let $0\le m_i\le n_i \le p-1$  be such that $m_i+n_i\le p-1$ for all $0\le i \le f-1.$ We have 
    \begin{eqnarray*}
        (m_0,...,m_i,...,m_{f-1}) &\otimes&(n_0,...,n_i,...,n_{f-1}) \\&\cong& \bigoplus_{k_{f-1}=0}^{m_{f-1}}\cdots\bigoplus_{k_0=0}^{m_0}(m_0+n_0-2k_0,...,m_{f-1}+n_{f-1}-2k_{f-1})\otimes \mathrm{det}^{\sum\limits_{i=0}^{f-1}k_ip^i}. 
    \end{eqnarray*} 
    
\end{corollary}
\begin{proof}
%    From Lemma \ref{singletupleidentity}, we have 
%    \begin{eqnarray*}
%       (0,\cdots ,m_i,\cdots,0)\otimes(0,\cdots ,n_i,\cdots ,0)&\cong&   (0,\cdots ,m_i-1,\cdots ,0)\otimes(0,\cdots ,n_i-1,\cdots ,0) \otimes \mathrm{det}^{p^i}\\
%       && \oplus (0,\cdots ,m_i + n_i,\cdots ,0).
%    \end{eqnarray*}
Repeatedly applying Lemma~\ref{singletupleidentity}, we obtain
    \begin{eqnarray*}
        (0,...,m_i,...,0) \otimes(0,...,n_i,...,0) &\cong& \bigoplus_{k_i=0}^{m_i}(0,...,m_i+n_i-2k_i,...,0)\otimes \mathrm{det}^{k_ip^i}. 
    \end{eqnarray*}
    Taking the tensor product over $i\in \{0,1,...,f-1\}$ yields the corollary. 
\end{proof}

\begin{comment}

    It is sufficient to show for $j \in \mathbb{Z}/f\mathbb{Z}$ that $me_j \otimes (p^k-1)e_{j-k} \cong ((m+1)p-1)e_{j-k}$. First we note that the map $\beta_{j,k}$ defined above is an injective equivariant map. Then we have the following sequence of homomorphisms.
    \[ me_j \otimes (p^k-1)e_{j-k} \xrightarrow{\beta_{j,k}\otimes\mathrm{Id}} mp^ke_{j-k}\otimes(p^k-1)e_{j-k} \xrightarrow{\varphi} ((m+1)p^k-1)e_{j-k}\]
    Where the map $\varphi$ takes $P \otimes Q$ to $PQ$. The image of $x_j^ly_j^{m-l}\otimes x^i_{j-k}y_{j-k}^{p^k-1-i} \in me_j\otimes (p^k-1)e_{j-k}$ under the composition of these maps is $x_{j-k}^{lp^k+i}y_{j-k}^{(m-l)p^k + p^k - i - 1}$. As $i$ and $l$ vary, this includes all the monomials in $((m+1)p^k-1)e_{j-k}$ so the composition is surjective. Also the composition of two equivariant map is equivariant. By comparing dimensions we see that it is an isomorphism.
\end{comment}

The following lemma generalizes Lemmas \ref{lem5} and \ref{lem6} and can be proved similarly. 

\begin{lemma}
     \label{largeweightsf}
     Let $0\le m_i\le n_i\le p-1$ be such that $p-2\le m_i+n_i\le 2p-2.$ We have 
     \begin{eqnarray*}
        m_ie_i\otimes n_ie_i&\cong& (p-m_i-2)e_i\otimes (p-n_i-2)e_i\otimes \mathrm{det}^{p^i(m_i+n_i+2-p)}\\
        &&\oplus~ (m_i+n_i + 1 - p)e_i\otimes (p-1)e_i.
     \end{eqnarray*}
\end{lemma}

The following result generalizes Corollary~\ref{largeweights}.

\begin{cor}
    Let $m = (m_0,\ldots,m_{f-1})$ and  $n = (n_0,\ldots,n_{f-1})$ with $0 \leq m_i \leq n_i \leq p-1$ be such that $p-2 \leq m_i + n_i \leq 2p-2$. For $l \in \{0,1\}^f$, let $\hat\psi_p(l) = \sum_i l_i (m_i + n_i + 2 - p)p^i$. We have
    \begin{align*}
    m \otimes n \cong \bigoplus_{l \in \{0,1\}^f}&\bigg[\big( l \odot ((p-2)\mathbbm{1}_f - m) + (\mathbbm{1}_f - l)\odot(m + n - (p-1)\mathbbm{1}_f)\big)\\
    &\otimes \big( l\odot ((p-2)\mathbbm{1}_f - n) + (\mathbbm{1}_f - l)\odot(p-1)\mathbbm{1}_f \big) \otimes \mathrm{det}^{\hat\psi_p(l)}\bigg].
    \end{align*}
\end{cor}
\begin{proof}
Let $h^j = \sum_{i = 0}^j e_i$. Let $W_j \subset \{0,1\}^f$ denote the subset consisting of elements of the form $(i_0,\ldots,i_j,0,\ldots,0)$.  We show by induction on $j$ that
 \begin{eqnarray*}
    m \otimes \sum_{i = 0}^jn_ie_i & \cong & \bigoplus_{l \in W_j} \bigg[\big(h^j\odot l \odot ((p-2)\mathbbm{1}_f - m) + h^j\odot(\mathbbm{1}_f - l)\odot(m + n - (p-1)\mathbbm{1}_f)\big)\\&&\quad \qquad \otimes \; ((\mathbbm{1}_f - h^j)\odot m)
    \otimes \big(h^j\odot l \odot ((p-2)\mathbbm{1}_f - n) + h^j \odot(\mathbbm{1}_f - l)\odot(p-1)\mathbbm{1}_f \big)\\ &&\qquad \quad \otimes \,\mathrm{det}^{\hat\psi_p(l)}\bigg].
\end{eqnarray*}
For $j = 0$, this is a consequence of the Lemma~\ref{largeweightsf}. Suppose we have the above formula for $j = t$. Then for $j = t+1$, we have
\begin{eqnarray*}
    m \otimes \sum_{i = 0}^{t+1}n_ie_i & \cong &\bigoplus_{l \in W_t} \bigg[\big(h^t\odot l \odot ((p-2)\mathbbm{1}_f - m) + h^t\odot(\mathbbm{1}_f - l)\odot(m + n - (p-1)\mathbbm{1}_f)\big)\\&&\quad \qquad \otimes \; ((\mathbbm{1}_f - h^t)\odot m)\otimes \big(h^t\odot l \odot ((p-2)\mathbbm{1}_f - n) + h^t \odot(\mathbbm{1}_f - l)\odot(p-1)\mathbbm{1}_f \big) \\&&\quad \qquad \otimes \; \mathrm{det}^{\hat\psi_p(l)} \bigg] \otimes n_{t+1}e_{t+1}.\\
\end{eqnarray*}
Writing $\mathbbm{1}_f - h^t = \mathbbm{1}_f - h^{t+1} + e_{t+1}$ and applying Lemma~\ref{largeweightsf}, we obtain 
\begin{eqnarray*}((\mathbbm{1}_f - h^t)\odot m)\otimes n_{t+1}e_{t+1} &\cong &((\mathbbm{1}_{f} - h^{t+1})\odot m) \otimes \bigg [ (p - m_{t+1} - 2)e_{t+1} \otimes (p - n_{t+1} - 2)e_{t+1}\\ && \quad \otimes \; \mathrm{det}^{p^{t+1}(m_{t+1} + n_{t+1} + 2 - p)}\oplus \big((m_{t+1} + n_{t+1} + 1 - p)e_{t+1} \otimes (p-1)e_{t+1}\big) \bigg]. 
\end{eqnarray*}
We also note that for $l \in W_{t}$, we have
{\tiny
\begin{eqnarray*}
    h^{t}\odot l \odot((p-2)\mathbbm{1}_f - m) \otimes (p - m_{t+1} - 2)e_{t+1} &\cong& h^{t+1} \odot (l + e_{t+1})\odot((p-2)\mathbbm{1}_f - m)\\
    h^{t}\odot l \odot((p-2)\mathbbm{1}_f - n) \otimes (p - n_{t+1} - 2)e_{t+1} &\cong& h^{t+1} \odot (l + e_{t+1})\odot((p-2)\mathbbm{1}_f - n)\\
    h^t\odot(\mathbbm{1}_{f} - l)\odot(p-1)\mathbbm{1}_f &\cong& h^{t+1}\odot(\mathbbm{1}_{f} - l-e_{t+1})\odot(p-1)\mathbbm{1}_f \\
    h^t\odot(\mathbbm{1}_{f} - l)\odot(p-1)\mathbbm{1}_f \otimes (p-1)e_{t+1} &\cong& h^{t+1}\odot(\mathbbm{1}_{f} - l)\odot(p-1)\mathbbm{1}_f \\
    %\end{eqnarray*}
    %\begin{equation*}
    h^{t}\odot (\mathbbm{1}_{f} - l)\odot( m  + n  - (p-1)\mathbbm{1}_f) & \cong  &  h^{t+1}\odot (\mathbbm{1}_{f} - l - e_{t+1})\odot( m  + n  - (p-1)\mathbbm{1}_f) \\
    %\end{equation*}
    %\begin{equation*}
        h^{t}\odot (\mathbbm{1}_{f} - l)\odot( m  + n  - (p-1)\mathbbm{1}_f)\otimes (m_{t+1} + n_{t+1} - (p-1))e_{t+1} & \cong &  h^{t+1}\odot (\mathbbm{1}_{f} - l)\odot( m  + n  - (p-1)\mathbbm{1}_f).
    \end{eqnarray*}
}    
\!\!\!Substituting these above, we get
{\small \begin{eqnarray*}
    m \otimes \sum_{i = 0}^{t+1}n_ie_i \cong & \bigoplus\limits_{l \in W_t}&\bigg[\big(h^{t+1}\odot (l + e_{t+1}) \odot ((p-2)\mathbbm{1}_f - m) \\
    &&\qquad \qquad + h^{t+1}\odot(\mathbbm{1}_f - l - e_{t+1})
    \odot (m + n - (p-1)\mathbbm{1}_f)\big) \\
    && \qquad \otimes \> ((\mathbbm{1}_f - h^{t+1})\odot m)\otimes \big(h^{t+1}\odot (l + e_{t+1})
    \odot((p-2)\mathbbm{1}_f - n) \\ 
    &&\qquad \qquad  + h^{t+1}  \odot(\mathbbm{1}_f - l - e_{t+1})\odot(p-1)\mathbbm{1}_f \big) \otimes \mathrm{det}^{\hat\psi_p(l + e_{t+1})}\\
    &&\quad  \oplus \> \big(h^{t+1}\odot l \odot ((p-2)\mathbbm{1}_f - m) \\
    &&\qquad \qquad + h^{t+1}\odot(\mathbbm{1}_f - l)
    \odot (m + n - (p-1)\mathbbm{1}_f)\big) \\
    &&\qquad \otimes \> ((\mathbbm{1}_f - h^{t+1})\odot m)\otimes \big(h^{t+1}\odot l
    \odot ((p-2)\mathbbm{1}_f - n) \\
    &&\qquad \qquad + h^{t+1} \odot(\mathbbm{1}_f - l)\odot(p-1)\mathbbm{1}_f \big) \otimes \mathrm{det}^{\hat\psi_p(l)} \bigg].
\end{eqnarray*}}
\!\!In the above expression we observe that the terms preceding the direct sum correspond to the ones given by those $l \in W_{t+1}$ whose $(t+1)$-th coordinate is $1$. The terms after the direct sum correspond to those $l \in W_{t+1}$ whose $(t+1)$-th coordinate is $0$. This proves the induction step and completes the proof of the corollary.
\end{proof}

Finally we treat the case where we take the tensor product with the symmetric power representations 
$(p^k-1,p^k-1,...,p^k-1)$. This result generalizes
Theorem~\ref{projective}.
\begin{theorem}
     Let $(m_0,m_1,\ldots,m_{f-1})$ be a representation of $\mathrm{GL}_2(\mathbb F_q).$ Let $0 \leq k \in \mathbb{Z}$. Then
    \[(m_0,\ldots,m_{f-1}) \otimes (p^k-1)\mathbbm{1}_f \cong ((m_{k}+1)p^k-1,\ldots,(m_{f-1}+1)p^k-1,(m_0+1)p^k-1,\ldots,(m_{k-1}+1)p^k -1).\]
\end{theorem}
\begin{proof}  We define a map $\beta_k:(m_0,m_1,...,m_{f-1})\rightarrow (m_kp^k,m_{k+1}p^k,...,m_0p^k,...,m_{k-1}p^k)$ given by 
$$P_0(x_0,y_0)\cdots P_{j}(x_j,y_j)\cdots P_{f-1}(x_{f-1},y_{f-1})\mapsto P_{0}(x_{f-k}^{p^k},y_{f-k}^{p^k})\cdots P_{j}(x_{j-k}^{p^k},y_{j-k}^{p^k})\cdots P_{f-1}(x_{f-1-k}^{p^k},y_{f-1-k}^{p^k})$$
with the convention that an index when negative is replaced by the congruent index mod $f$ with representative in $[0,f-1]$.
Clearly, $\beta$ is an injective linear map. We check that it is $\mathrm{GL}_2(\Fq)$-equivariant. Let $\alpha = \left( \begin{smallmatrix} a & b \\ c & d \end{smallmatrix} \right) \in \mathrm{GL}_2(\Fq)$. Then

\begin{align*}
 \beta_{k}(\alpha\cdot \prod_{j=0}^{f-1}P_j( x_j,y_j)) &= \beta_{k}(\prod_{j=0}^{f-1}P_j(a^{p^j} x_j + c^{p^j}y_j,b^{p^j} x_j + d^{p^j}y_j)\\
 &=\prod_{j=0}^{f-1} P_j(a^{p^j} x^{p^k}_{j-k} + c^{p^j}y^{p^k}_{j-k},b^{p^j} x^{p^k}_{j-k} + d^{p^j}y^{p^k}_{j-k})
 \\
 &=\prod_{j=0}^{f-1} P_j((a^{p^{j-k}} x_{j-k} + c^{p^{j-k}}y_{j-k})^{p^k},(b^{p^{j-k}} x_{j-k} + d^{p^{j-k}}y_{j-k})^{p^k})\\
 &= \alpha \cdot \prod_{j=0}^{f-1}P_j(x_{j-k}^{p^k},y_{j-k}^{p^k})\\
 &= \alpha\cdot\beta_{k}(P(x_j,y_j)).
\end{align*}
Now we define the following sequence of homomorphisms:
\begin{align*}
   (&m_0,m_1,\ldots,m_{f-1}) \otimes (p^k-1)\mathbbm 1_f \xrightarrow{\beta_{k}\otimes\mathrm{Id}} (m_kp^k,m_{k+1}p^k,...,m_0p^k,...,m_{k-1}p^k)\otimes (p^k-1)\mathbbm 1_f \\&\xrightarrow{\varphi} ((m_{k}+1)p^k-1,\ldots,(m_{f-1}+1)p^k-1,(m_0+1)p^k-1,\ldots,(m_{k-1}+1)p^k -1),
\end{align*} 
where the second map $\phi$ is given by $P\otimes Q\mapsto PQ.$ Since both the maps in the above sequence are $\mathrm{GL}_2(\Fq)$-equivariant, the composition is also $\mathrm{GL}_2(\Fq)$-equivariant. The image of any monomial $\prod_{j=0}^{f-1}x_{j}^{l_j}y_{j}^{m_j-l_j}\otimes \prod_{j=0}^{f-1}x_{j}^{s_j}y_{j}^{p^k-1-s_j}$ under the composition map is given by $$\prod_{j=0}^{f-1}x_{j-k}^{l_{j}p^k+s_{j-k}}y_{j-k}^{(m_{j}-l_j)p^k+p^k-1-s_{j-k}}.$$
For any $j,$ as $l_j$ varies from $0$ to $m_j$ and as $s_{j-k}$ varies from $0$ to $p^k-1$, we get that $l_{j}p^k+s_{j-k}$ varies from $0$ to $(m_{j}+1)p^k-1.$ Thus the composition map is surjective. Now a comparison of the dimension of the two spaces shows that this surjection is an isomorphism.
\end{proof}

\section{Structure of $V_r/V_r^{(m+1)}$}

\subsection{Application of Clebsch-Gordan to the structure of $V_r/V_r^{**}$}

By Rozensztajn \cite{sandra} for $f = 1$ and Ghate-Jana \cite{GhateJana+2025+1503+1543} for general $f \geq 1$, we know that $V_r/V_r^*$ is a principal series. One may ask about the structure of $V_r/V_r^{**}$. For $f = 1$, it is well known to be an extension of principal series. Here we use the Clebsch-Gordan decompositions in the previous section to investigate the case of $f = 2$. 

Let $G = \mathrm{GL}_2$ and let $B$ be the subgroup of upper triangular matrices. By \cite{GhateJana+2025+1503+1543}, Theorem 1.3,  if \(p\nmid r_0,\ p\nmid r_1\), then 
\[\frac{V_r}{V_r^{**}}\cong %\text{ind}_{B(\mathbb{F}_q)}^{G(\mathbb{F}_q)} d^{(r_0-1)+p(r_1-1)}\otimes(1,1)\right)\cong  
\text{ind}_{B(\mathbb{F}_q)}^{G(\mathbb{F}_q)} \> d^{r_0-1+p(r_1-1)} \otimes(1,1).
\]
%where in the second isomorphism, since $(1,1)$ is a representation of $\mathrm{GL}_2(\Fq),$ the tensor commutes with the ind functor. 
Write \(r_0-1+p(r_1-1)\equiv a= a_0 + pa_1\) mod \((p^2-1)\), where \(0\le a_i<p\).  By Breuil's Columbia
notes \cite[Theorem 7.6]{Breuil} (see also Breuil-Pa{\v{s}}k{\={u}}nas \cite{BP} and Diamond \cite{diamond}), we conclude that if \(a \notin \{ 0,p^2-1\}\), then \(\text{ind}_{B(\mathbb{F}_q)}^{G(\mathbb{F}_q)} \> d^a
\) has four Jordan-H\"older factors (weights) whose socle filtration is given by the following diagram:
\[
\begin{tikzpicture}[>=stealth, x=4cm, y=-1.5cm]
  % Nodes
  \node (A) at (1,.5) {\text{\shortstack{$(p{-}1{-}a_0,\,p{-}1{-}a_1)\otimes D^a$}}};
  \node (B) at (2,1.5) {\text{\shortstack{$(p{-}2{-}a_0,\,a_1{-}1)\otimes D^{1+a_0}$}}};
  \node (C) at (0,1.5) {\text{\shortstack{$(a_0{-}1,\,p{-}2{-}a_1)\otimes D^{(1+a_1)p}$}}};
  \node (D) at (1,2.5)   {\text{$(a_0,a_1)$}};

  % Arrows (no heads)
  \draw[-] (A) -- (B);
  \draw[-] (A) -- (C);
  \draw[-] (B) -- (D);
  \draw[-] (C) -- (D);
\end{tikzpicture}
\]
where we write $D = \det$ for ease of notation.
%The structure of \(V_r/V_r^{**}\) is given by the tensor %product of each Jordan-H\"older factor of 
%\(\text{ind}_{B(\mathbb{F}_q)}^{G(\mathbb{F}_q)}\left(d^{a}\right)
%\) with \((1,1)\) as shown in the following diagram:
%\[
%\begin{tikzpicture}[>=stealth, x=4cm, y=-1.5cm]
%  % Nodes
%  \node (A) at (1,0) {\text{\shortstack{$(p-1{-}a_0,\,p{-}1{-}a_1)\otimes D^a\otimes(1,1)$}}};
%  \node (B) at (0,1.5) {\text{\shortstack{$(p{-}2{-}a_0,\,a_1{-}1)\otimes D^{1+a_0}\otimes(1,1)$}}};
%  \node (C) at (2,1.5) {\text{\shortstack{$(a_0{-}1,\,p{-}2{-}a_1)\otimes D^{(1+a_1)p}\otimes(1,1)$}}};
%  \node (D) at (1,3) {\text{\shortstack{$(a_0,a_1)\otimes(1,1)$}}};

  % Arrows (no head)
%  \draw[-] (A) -- (B);
%  \draw[-] (A) -- (C);
%  \draw[-] (B) -- (D);
%  \draw[-] (C) -- (D);
%\end{tikzpicture}
%\]

If \(a_0,a_1\notin \{0,1,p-2,p-1\}\), then tensoring each of
the four terms above with $(1,1)$, by the Clebsch-Gordan formula in Corollary~\ref{cgtheorem}, we obtain
the following sixteen Jordan-H\"older factors in  \(V_r/V_r^{**}\):

{
\[\resizebox{.8\textwidth}{!}{
\begin{tikzpicture}[>=stealth, x=5cm, y=-2cm]

  % === Nodes ===
  \node (A) at (1,1) {\text{\shortstack[l]{%
    $(p{-}2{-}a_0,\,p{-}2{-}a_1)\otimes D^{a+p+1} \oplus(p{-}2{-}a_0,\,p{-}a_1)\otimes D^{a+1} $\\ $\oplus\;(p{-}a_0,\,p{-}2{-}a_1)\otimes D^{a+p} \oplus(p{-}a_0,\,p{-}a_1)\otimes D^{a}$  }}};

 \node (B) at (2,2) {\text{\shortstack[l]{%
    $(p{-}3{-}a_0,\,a_1{-}2)\otimes D^{2+a_0+p} $\\
    $\oplus\;(p{-}3{-}a_0,\,a_1)\otimes D^{2+a_0} $\\
    $\oplus\;(p{-}1{-}a_0,\,a_1{-}2)\otimes D^{1+a_0+p} $\\
    $\oplus\;(p{-}1{-}a_0,\,a_1)\otimes D^{1+a_0}$
  }}};

  \node (C) at (0,2) {\text{\shortstack[l]{%
    $(a_0{-}2,\,p{-}3{-}a_1)\otimes D^{(2+a_1)p+1} $\\
    $\oplus\;(a_0{-}2,\,p{-}1{-}a_1)\otimes D^{(1+a_1)p+1} $\\
    $\oplus\;(a_0,\,p{-}3{-}a_1)\otimes D^{(2+a_1)p} $\\
    $\oplus\;(a_0,\,p{-}1{-}a_1)\otimes D^{(1+a_1)p}$
  }}};

  \node (D) at (1,3) {\text{\shortstack[l]{%
    $(a_0{-}1,\,a_1{-}1)\otimes D^{p+1} \oplus (a_0{-}1,\,a_1{+}1)\otimes D$\\
    $ \oplus\;(a_0{+}1,\,a_1{-}1)\otimes D^p \oplus(a_0{+}1,\,a_1{+}1).$
    }}};
  % === Arrows ===
  \draw[-] (A) -- (B);
  \draw[-] (A) -- (C);
  \draw[-] (B) -- (D);
  \draw[-] (C) -- (D);
\end{tikzpicture}}\]
}
%\end{comment}

Theorem 11.4 in \cite{Breuil} (see also \cite{BP}) gives a characterization of the cases where two weights can have a non-split extension. It turns out that most of the extensions in the above diagram are split. Identifying the possibly non-split extensions between the weights, the possibility of the following four principal series inside \(V_r/V_r^{**}\) emerges:

\[
\begin{tikzpicture}[>=stealth, x=4cm, y=-1.5cm]
  % Nodes
  \node (A) at (1,0.5) {\text{\shortstack{$(p-a_0,\,p-a_1)\otimes D^{a}$}}};
  \node (B) at (0,1.5) {\text{\shortstack{${(a_0-2,p-1-a_1) \otimes D^{(1+a_1)p+1}}$}}};
  \node (C) at (2,1.5) {\text{\shortstack{$(p-1-a_0,a_1-2)\otimes D^{1+a_0+p}$}}};
  \node (D) at (1,2.5)   {\text{$(a_0-1,a_1-1)\otimes D^{p+1}.$}};

  % Arrows (no heads)
  \draw[-] (A) -- (B);
  \draw[-] (A) -- (C);
  \draw[-] (B) -- (D);
  \draw[-] (C) -- (D);
\end{tikzpicture}
\]

\[
\begin{tikzpicture}[>=stealth, x=4cm, y=-1.5cm]
  % Nodes
  \node (A) at (1,.5) {\text{\shortstack{$(p-a_0,\,p{-}2{-}a_1)\otimes D^{a+p}$}}};
  \node (B) at (0,1.5) {\text{\shortstack{${(a_0-2,p-3-a_1) \otimes D^{(2+a_1)p+1}}$}}};
  \node (C) at (2,1.5) {\text{\shortstack{$(p-1-a_0,\,a_1)\otimes D^{1+a_0}$}}};
  \node (D) at (1,2.5)   {\text{$(a_0-1,a_1+1)\otimes D$}};
  % Arrows (no heads)
  \draw[-] (A) -- (B);
  \draw[-] (A) -- (C);
  \draw[-] (B) -- (D);
  \draw[-] (C) -- (D);
\end{tikzpicture}
\]

\[
\begin{tikzpicture}[>=stealth, x=4cm, y=-1.5cm]
  % Nodes
  \node (A) at (1,.5) {\text{\shortstack{$(p-2-a_0,\,p{-}a_1)\otimes D^{a+1}$}}};
  \node (B) at (0,1.5) {\text{\shortstack{${(a_0,p-1-a_1) \otimes D^{(1+a_1)p}}$}}};
  \node (C) at (2,1.5) {\text{\shortstack{$(p-3-a_0,a_1-2)\otimes D^{2+a_0+p}$}}};
  \node (D) at (1,2.5)   {\text{$(a_0+1,a_1-1)\otimes D^p$}};

  % Arrows (no heads)
  \draw[-] (A) -- (B);
  \draw[-] (A) -- (C);
  \draw[-] (B) -- (D);
  \draw[-] (C) -- (D);
\end{tikzpicture}
\]

\[
\begin{tikzpicture}[>=stealth, x=4cm, y=-1.5cm]
  % Nodes
  \node (A) at (1,.5) {\text{\shortstack{$(p-2-a_0,\,p{-}2{-}a_1)\otimes D^{a+p+1}$}}};
  \node (B) at (0,1.5) {\text{\shortstack{${(a_0,p-3-a_1) \otimes D^{(2+a_1)p}}$}}};
  \node (C) at (2,1.5) {\text{\shortstack{$(p-3-a_0,a_1)\otimes D^{2+a_0}$}}};
  \node (D) at (1,2.5)   {\text{$(a_0+1,a_1+1)$}};

  % Arrows (no heads)
  \draw[-] (A) -- (B);
  \draw[-] (A) -- (C);
  \draw[-] (B) -- (D);
  \draw[-] (C) -- (D);
\end{tikzpicture}
\]

At this stage it is not clear how these four possible principal series are arranged in \(V_r/V_r^{**}\). For \(f=1\), \(V_r/V_r^*\) is an extension between two weights and \(V_r/V_r^{**}\) is an extension between two principal series. One might expect something similar to happen for \(f=2\). Since for \(f=2\) the quotient 
\(V_r/V_r^*\) is a diamond shaped diagram of four weights, one might expect \(V_r/V_r^{**}\) to also be a diamond shaped diagram with the four weights replaced by four principal series, as in the following conjectural diagram:
\[\begin{tikzcd}
	&&& {[4,1]} \\
	&& {[2,3]} && {[3,2]} \\
	&&& {[1,4]} \\
	& {[4,3]} &&&& {[4,2]} \\
	{[2,1]} && {[3,4]} && {[2,4]} && {[3,1]} \\
	& {[1,2]} &&&& {[1,3]} \\
	&&& {[4,4]} \\
	&& {[2,2]} && {[3,3]} \\
&&& {[1,1].}
	\arrow[no head, from=1-4, to=2-3]
	\arrow[no head, from=1-4, to=2-5]
	\arrow[no head, from=2-3, to=3-4]
	\arrow[no head, from=2-5, to=3-4]
	\arrow[no head, from=4-2, to=5-1]
	\arrow[no head, from=4-2, to=5-3]
	\arrow[no head, from=4-6, to=5-5]
	\arrow[no head, from=4-6, to=5-7]
	\arrow[no head, from=5-1, to=6-2]
	\arrow[no head, from=5-3, to=1-4]
	\arrow[no head, from=5-3, to=6-2]
	\arrow[no head, from=5-5, to=1-4]
	\arrow[no head, from=5-5, to=6-6]
	\arrow[no head, from=5-7, to=6-6]
	\arrow[no head, from=6-2, to=2-3]
	\arrow[no head, from=6-6, to=2-5]
	\arrow[no head, from=7-4, to=8-3]
	\arrow[no head, from=7-4, to=8-5]
	\arrow[no head, from=8-3, to=4-2]
	\arrow[no head, from=8-3, to=9-4]
	\arrow[no head, from=8-5, to=4-6]
	\arrow[no head, from=8-5, to=9-4]
	\arrow[no head, from=9-4, to=5-3]
	\arrow[no head, from=9-4, to=5-5]
\end{tikzcd}\]
Here we use some new notation: $[i,j]$ is the weight in the $i$-th row and $j$-th column of the 
following table:
{\tiny
\begin{center}
\resizebox{1\textwidth}{!}{
\begin{tabular}{ |c|c|c|c|c| } 
 \hline
 & 1 & 2 & 3 & 4 \\ \hline
 1 & $\begin{array}{c}(a_0{-}1,\ a_1{-}1)\\ \otimes D^{p+1}\end{array}$ &
     $\begin{array}{c}(a_0{-}1,\ a_1{+}1)\\ \otimes D^{1}\end{array}$  &
     $\begin{array}{c}(a_0{+}1,\ a_1{-}1)\\ \otimes D^{p}\end{array}$ &
     $(a_0{+}1,\ a_1{+}1)$ \\ 
 \hline
 2 & $\begin{array}{c}(a_0{-}2,\ p{-}3{-}a_1)\\ \otimes D^{(2{+}a_1)p+1}\end{array}$ &
     $\begin{array}{c}(a_0{-}2,\ p{-}1{-}a_1)\\ \otimes D^{(1{+}a_1)p+1}\end{array}$  &
     $\begin{array}{c}(a_0,\ p{-}3{-}a_1)\\ \otimes D^{(2{+}a_1)p}\end{array}$ &
     $\begin{array}{c}(a_0,\ p{-}1{-}a_1)\\ \otimes D^{(1{+}a_1)p}\end{array}$ \\ 
 \hline
 3 & $\begin{array}{c}(p{-}3{-}a_0,\ a_1{-}2)\\ \otimes D^{2+a_0+p}\end{array}$ &
     $\begin{array}{c}(p{-}3{-}a_0,\ a_1)\\ \otimes D^{2+a_0}\end{array}$  &
     $\begin{array}{c}(p{-}1{-}a_0,\ a_1{-}2)\\ \otimes D^{1+a_0+p}\end{array}$ &
     $\begin{array}{c}(p{-}1{-}a_0,\ a_1)\\ \otimes D^{1+a_0}\end{array}$ \\ 
 \hline
 4 & $\begin{array}{c}(p{-}2{-}a_0,\ p{-}2{-}a_1)\\ \otimes D^{a{+}p{+}1}\end{array}$ &
     $\begin{array}{c}(p{-}2{-}a_0,\ p{-}a_1)\\ \otimes D^{a{+}1}\end{array}$  &
     $\begin{array}{c}(p{-}a_0,\ p{-}2{-}a_1)\\ \otimes D^{a{+}p}\end{array}$ &
     $\begin{array}{c}(p{-}a_0,\ p{-}a_1)\\ \otimes D^{a}\end{array}$ \\ 
 \hline
\end{tabular}
}
\end{center}
}
\vspace{.2cm}

However, it is not clear how to check that the above arrangement of principal series representations in $V_r/V_r^{**}$ for $f = 2$  is correct with the present tools. In the next subsections, we use another method to study the structure of $V_r/V_r^{(m+1)}$ for general $m \geq 0$ and $f \geq 1$ which uses the theta filtration. 
%In the following sub-sections of Section 3 we prove the arrangement of principal series representation and in Section 4, we prove the correspondence between the Jordan H\"older factors of \(V_r/V_r^*\) and the principal series present in \(V_r/V_r^{**}\) 

%To get the sub-quotients in \(V_r/V_r^{(m)} \), we will define a sub-module filtration, such that in the appropriate quotient spaces, we will get a principal series representation. 
%\cite{Breuil}
\subsection{Theta Filtration for $f=2$}

It is illuminating to do this first for $V_r/V_r^{**}$ and $f=2.$ Consider the lattice of submodules
of $V_r/V_r^{**}$ given in the picture:

\[
\begin{tikzpicture}[>=stealth, x=4cm, y=-1.5cm]
  % Nodes
  \node (A) at (1,0) {\text{\shortstack{$\frac{V_r^*}{V_r^{**}}$}}};
  \node (B) at (0,1.5) {\text{\shortstack{$\frac{\langle \theta_0\rangle+V_r^{**}}{V_r^{**}}$}}};
  \node (C) at (2,1.5) {\text{\shortstack{$\frac{\langle \theta_1\rangle+V_r^{**}}{V_r^{**}}$}}};
  \node (D) at (1,3)   {\text{$\frac{\langle \theta_0\theta_1\rangle+V_r^{**}}{V_r^{**}}$}};
  \node (Top) at (1,4) {\text{0}.};
  \node (Socle) at (1,-1) {\text{$\frac{V_r}{V_r^{**}}$}};

  % Arrows
  \draw[-] (A) -- (B);
  \draw[-] (A) -- (C);
  \draw[-] (B) -- (D);
  \draw[-] (C) -- (D);

  % Top and bottom arrows with objects
  \draw[-] (D) -- (Top);
  \draw[-] (Socle) -- (A);
\end{tikzpicture}
\]
We show that the sub-quotients in the above diagram consist of principal series arranged in a diamond shaped diagram. The top sub-quotient is $\frac{V_r}{V_r^{**}}/\frac{V_r^*}{V_r^{**}}\cong \frac{V_r}{V_r^*},$ which is a principal series by \cite{GhateJana+2025+1503+1543}. Now we study the quotient $$\frac{\frac{V_r^*}{V_r^{**}}}{\frac{\langle \theta_0\rangle +V_r^{**}}{V_r^{**}}}\cong \frac{\langle\theta_0,\theta_1\rangle}{\langle \theta_0,\theta_1^2\rangle}\cong \frac{\langle \theta_1\rangle}{\langle\theta_0\theta_1,\theta_1^2\rangle}.$$
The second isomorphism follows from the second isomorphism theorem and a small check using the fact that $\theta_1 \nmid \theta_0$. We claim that the rightmost quotient is a principal series.
Let $r'=(r_0-p,r_1-1).$  Consider the map $$V_{r'}\otimes \mathrm{det}^p\rightarrow \frac{\langle \theta_1\rangle}{\langle \theta_0\theta_1,\theta_1^2\rangle}$$
given by multiplication by $\theta_1 = x_1y_0^p -y_1 x_0^p$.  This map is $\mathrm{GL}_2(\mathbb F_q)$-equivariant and surjective with 
\begin{eqnarray*}
    \mathrm{kernel}&=&\{P:P\theta_1 =A\theta_0\theta_1+B\theta_1^2\}\\
    &=&\{P:P =A\theta_0+B\theta_1\}\\
    &=&\langle\theta_0,\theta_1\rangle.\end{eqnarray*}
    Thus we have $\frac{V_{r'}}{V_{r'}^*}\otimes \mathrm{det}^p\cong \frac{\langle \theta_1\rangle}{\langle \theta_0\theta_1,\theta_1^2\rangle}$. 
    Since the quotient on the left is a principal series by \cite{GhateJana+2025+1503+1543}, we are done. Now we study the quotient $$\frac{\frac{\langle\theta_0\rangle+V_r^{**}}{V_r^{**}}}{\frac{\langle\theta_0\theta_1\rangle+V_r^{**}}{V_r^{**}}}\cong\frac{\langle\theta_0\rangle+V_r^{**}}{\langle\theta_0\theta_1\rangle+V_r^{**}}\cong \frac{\langle \theta_0\rangle}{\langle \theta_0\theta_1,\theta_0^2\rangle}.$$
Now let $r'' =(r_0-1,r_1-p).$ Consider the map $$V_{r''}\otimes \mathrm{det}\rightarrow \frac{\langle \theta_0\rangle}{\langle \theta_0\theta_1,\theta_0^2\rangle}$$
given by multiplication by $\theta_0 =  x_0y_1^p -y_0 x_1^p$. The map is $\mathrm{GL}_2(\mathbb F_q)$-equivariant and surjective with 
\begin{eqnarray*}
    \mathrm{kernel}&=&\{P:P\theta_0 =A\theta_0\theta_1+B\theta_0^2\}\\
    &=&\{P:P =A\theta_1+B\theta_0\}\\
    &=&\langle\theta_0,\theta_1\rangle.
\end{eqnarray*}
Thus $\frac{\langle \theta_0\rangle}{\langle \theta_0\theta_1,\theta_0^2\rangle}\cong \frac{V_{r''}}{V_{r''}^{*}}\otimes \mathrm{det}.$ Again by \cite{GhateJana+2025+1503+1543}, the quotient is a principal series. Finally, we study the quotient $\frac{\langle \theta_0\theta_1\rangle+V_r^{**}}{V_r^{**}}\cong \frac{\langle \theta_0\theta_1\rangle}{\langle \theta_0^2\theta_1,\theta_0\theta_1^2\rangle}.$ Consider the map $$V_{r'''}\otimes \mathrm{det}^{1+p}\rightarrow\frac{\langle\theta_0\theta_1\rangle}{\langle \theta_0^2\theta_1,\theta_0\theta_1^2\rangle}$$ given by multiplication by $\theta_0\theta_1$, where $r'''=(r_0-p-1,r_1-p-1).$ The twist by $\mathrm{det}^{1+p}$ makes the map $\mathrm{GL}_2(\mathbb F_q)$-equivariant. Clearly the map is surjective with 
\begin{eqnarray*}
    \mathrm{kernel}&=&\{P:P\theta_0\theta_1 =A\theta_0^2\theta_1+B\theta_0\theta_1^2\}\\
    &=&\{P:P =A\theta_0+B\theta_1\}\\
    &=&\langle\theta_0,\theta_1\rangle.
\end{eqnarray*}
Thus we have $\frac{\langle\theta_0\theta_1\rangle}{\langle \theta_0^2\theta_1,\theta_0\theta_1^2\rangle}\cong\frac{V_{r'''}}{V_{r'''}^*}\otimes \mathrm{det}^{1+p}$, which is a principal series, by \cite{GhateJana+2025+1503+1543}.  
This shows that $V_r/V_r^{**}$ has a filtration of four submodules (given by the left side of the diamond above) with each sub-quotient a principal series. 

Similarly, one can prove that the sub-quotients on the right are principal series as well. Thus, the principal series in $V_r/V_r^{**}$ are indeed arranged in a diamond shaped diagram when $f = 2$. In the following subsections, we study the structure of \(V_r/V_r^{(m+1)}\) for arbitrary \(m\) and \(f\), generalizing this argument.

\subsection{Some isomorphisms}
In this subsection we prove some isomorphisms that will be used in the proof of the main theorem. In the following, isomorphism means isomorphism as representations of GL\(_2(\mathbb{F}_{q})\). Let \(r=\sum_{i=0}^{f-1}r_ip^i\) and \(V_r=\otimes_{i=0}^{f-1}(\mathrm{Sym}^{r_i} {\mathbb F}_q^2 \circ \text{Fr}^i)\).
For any polynomial \(f \in \mathbb{F}_q[x_0,y_0,...,x_{f-1},y_{f-1}]\), let \(\langle f \rangle\) denote the submodule of \(V_r\) consisting of all the polynomials in \(V_r\) which are divisible by \(f\). If there are multiple polynomials \(f_1,...,f_k\), then \(\langle f_1,...,f_k \rangle:=\langle f_1\rangle + \langle f_2\rangle + ... + \langle f_k\rangle\). Also, for any submodule \(V\subset V_r\), let \[[V]:=\frac{V+V_r^{(m+1)}}{V_r^{(m+1)}}\] 
denote the submodule of $V_r/V_r^{(m+1)}$ generated by $V$.
\begin{comment}
\begin{lemma} \label{span_lemma} For \(1\le l\le n,\ 1\le t\le s\), let \(f_{l,t} \in \mathbb{F}_q[X_0,Y_0,...,X_{f-1},Y_{f-1}]\). Then,
    \[
    \sum_{l=0}^{n}\left[ \cup_{t=0}^{s}\langle f_{l,t} \rangle \right] \cong \left[ \langle \cup_{l=0}^n\cup_{t=0}^s\{f_{l,t}\} \rangle \right]
    \]
    In particular, if there exists \(l_1,\ t_1,\ l_2,\ t_2\) such that \(f_{l_1,t_1}|f_{l_2,t_2}\), then \(f_{l_2,t_2}\) can be removed from the union in RHS. 
\end{lemma}
\begin{proof}
Lemma is true because both the subspaces are generated by $ f_{l,t}$ for $l=0,1,...,n$ and $t=0,1,...,s.$    
\end{proof}
    
\end{comment}

\begin{comment}
\begin{lemma}\label{iso2}  For any subspace \(V\subset V_r\)
    \[
    [V]\cong \frac{V+V_r^{(m+1)}}{V_r^{(m+1)}}
    \]
\end{lemma}
\begin{proof}
    Since $[V]=V/V\cap V_r^{(m+1)},$ the lemma follows from second vector space isomorphism theorem.
\end{proof}
\end{comment}
\begin{lemma} \label{iso3} For any submodules \(W\subset V\subset V_r\)
   $$ \frac{[V]}{[W]}\cong \frac{V}{W+ V\cap V_r^{(m+1)}}.$$
    \end{lemma}
\begin{proof}By definition of $[V]$ and $[W]$, we have 
    $$\frac{[V]}{[W]}\cong \frac{V+V_r^{(m+1)}}{W+V_r^{(m+1)}}.$$
    Define the map  $V \rightarrow \frac{V+V_r^{(m+1)}}{W+V_r^{(m+1)}}$ by $v \mapsto (v +V_r^{(m+1)})+W +V_r^{(m+1)}.$ Clearly this map is $\mathrm{GL}_2({\mathbb F}_q)$-equivariant and surjective. Moreover, it has
    \begin{eqnarray*}
        \mathrm{kernel} &=&\{v\in V \mid v +V_r^{(m+1)} \in W+ V_r^{(m+1)}  \}\\
        &=&\{v\in V \mid v =w +x, w \in W, x \in  V_r^{(m+1)}\}\\
       % &=&\{v\in V \mid v =w +x, w \in W, x \in  V_r^{(m+1)}, x \in V\}\\
        &=&\{v\in V \mid v =w +x, w \in W, x \in  V_r^{(m+1)}\cap V\} \quad \text{since } W \subset V\\
        &=&W+V\cap V_r^{(m+1)}. 
    \end{eqnarray*}
    Hence the lemma follows.
\end{proof}

The main tool we use in the proof of Proposition~\ref{iso1} below is that $V_r/V_r^{*}=V_r/V_r^{(1)}$
is a principal series \cite{sandra}, \cite{GhateJana+2025+1503+1543}. 
We will later show that each of the sub-quotients in the theta filtration is isomorphic to a representation of the form in Proposition \ref{iso1}. 
%This will prove that every sub-quotient is a principal series. 

\begin{proposition} For any \(j_i\ge0 \text{ and } i=0,1,...,f-1\),
    \label{iso1} 
    \[
        \frac{\Big\langle \prod_{l=0}^{f-1}\theta_l^{j_l} \Big\rangle}
        {\Big\langle \cup_{i=0}^{f-1}\left\{\left(\prod_{l=0;\ l\ne i}^{f-1} \theta_l^{j_l}\right)\left(\theta_i^{j_i+1}\right) \right\} \Big\rangle}
        \cong \mathrm{ind}_{B(\mathbb F_q)}^{G(\mathbb F_q)}\left(\mathrm{det}^{S_P}\otimes d^{r'}\right)
    \]
    is a principal series, 
    where \(r'=\sum_{i=0}^{f-1}r'_ip^i\) with  \(r'_i=r_i-j_i-pj_{i+1} \geq q\), and \(S_P= \sum_{l=0}^{f-1}j_l\ p^{l}\).  
\end{proposition}

\begin{proof}
First we look at the action of a matrix $\left(\begin{smallmatrix}
    a&b\\c&d
\end{smallmatrix}\right)$  in GL$_2(\mathbb F_q)$ on $\theta_i=x_iy_{i-1}^p-y_ix_{i-1}^p.$ Then working mod $p$ we have
\begin{eqnarray*}
\left(\begin{smallmatrix}
    a&b\\c&d
\end{smallmatrix}\right)\cdot \theta_i&=&  (a^{p^i}x_i+c^{p^i}y_i)(b^{p^{i-1}}x_{i-1}+d^{p^{i-1}}y_{i-1})^p-(b^{p^i}x_i+d^{p^i}y_i)(a^{p^{i-1}}x_{i-1}+c^{p^{i-1}}y_{i-1})^p \\
&=&(a^{p^i}x_i+c^{p^i}y_i)(b^{p^{i}}x^p_{i-1}+d^{p^{i}}y^p_{i-1})-(b^{p^i}x_i+d^{p^i}y_i)(a^{p^{i}}x^p_{i-1}+c^{p^{i}}y^p_{i-1})\\
&=&((ad)^{p^i}-(bc)^{p^i})x_iy_{i-1}^p-((ad)^{p^i}-(bc)^{p^i})y_ix_{i-1}^p\\
&=&(ad-bc)^{p^i}\theta_i.
\end{eqnarray*}
Let \(P=\prod_{l=0}^{f-1}\theta_l^{j_l}\).
    Since $\mathrm{GL}_2(\Fq)$ acts on $\theta_l$ by $\mathrm{det}^{p^l}$, it acts on $P$ by the character $\mathrm{det}^{S_P}$.
    
Let
    \begin{align*}
        &V':=\Bigg\langle \bigcup_{i=0}^{f-1}\left\{\left(\prod_{l=0;\ l\ne i}^{f-1} \theta_l^{j_l}\right)\left(\theta_i^{j_i+1}\right) \right\} \Bigg\rangle 
        \subset V_r.
    \end{align*} 
Let \(
        V_{r'}:=\otimes_{i=0}^{f-1}V_{r'_i}\circ \text{Fr}^i\). 
    Define the map \(\psi:\det^{S_P}\otimes V_{r'}\to V_r/V'\),
    \[
        \psi(Q)=PQ+V'.
    \]
    Notice that twisting by the character $\mathrm{det} ^{S_P}$ makes \(\psi\) a GL$_2(\mathbb F_q)$-equivariant map. Clearly, the image of \(\psi\) is \(\langle P\rangle/V' \subset V_r/V'\). Now we compute the kernel of $\psi.$ Our claim is that kernel of \(\psi\) is \( V_{r'}^* =\langle \theta_0,...,\theta_{f-1} \rangle \subset V_{r'}\).
    Clearly $V_{r'}^* \subset \ker (\psi).$ Let $Q \in \ker (\psi).$ We have 
    $$QP= \sum\limits_{i=0}^{f-1}A_i\left(\prod_{l=0;\ l\ne i}^{f-1} \theta_l^{j_l}\right)\left(\theta_i^{j_i+1}\right). $$
Dividing both sides of this equation by $P$, we obtain
$$Q=\sum\limits_{i=0}^{k-1}A_i\theta_i\in V_{r'}^*.$$
Hence $\ker(\psi)=V_{r'}^*.$
Thus we have
    \[
        \frac{\langle P \rangle}{V'}\cong \mathrm{det}^{S_P} \otimes\frac{V_{r'}}{V_{r'}^*}.    
    \]    
    By Theorem 1.3 in \cite{GhateJana+2025+1503+1543}, the right hand side is isomorphic to \(\text{ind}_{B(\mathbb{F}_q)}^{G(\mathbb{F}_q)}\left( \text{det}^{S_P}\otimes d^{r'}\right)\).
\end{proof}

\subsection{Main Theorem}
\begin{definition}
We say that a representation \(V\) is decomposable into principal series if it is possible to write down a filtration
of submodules such that the successive sub-quotients are principal series.
%to partition the set of irreducible sub-quotients of \(V\) into disjoint subsets \(S_i\), such that the sub-quotients in each \(S_i\) form a principal series. 
\end{definition}

The main result in this section is that \(V_r/V_r^{(m+1)}\) decomposes into principal series. We will use induction on $m$ to prove this. We will assume that $V_r/V_r^{(m)}$ decomposes into principal series and prove that $V_r^{(m)}/V_r^{(m+1)}$ is decomposable into principal series. This will show that $V_r/V_r^{(m+1)}$ is decomposable into principal series and complete the inductive step. 

As in the case of $m = 1$ and $f = 2$, we start with the lattice of submodules in $V_r/V_r^{(m+1)}$ generated by all products of all powers of $\theta_i$.  We call this the {\it theta filtration} on $V_r/V_r^{(m+1)}$. It forms a hypercube graph. For instance,
the theta filtration on 
$V_r/V_r^{(m+1)}$ for \(m=2\) and \(f=3\) is given by 
the following picture where as before we write $[V]=\frac{V+V_r^{(m+1)}}{ V_r^{(m+1)}}$ for a submodule
$V \subset V_r$.

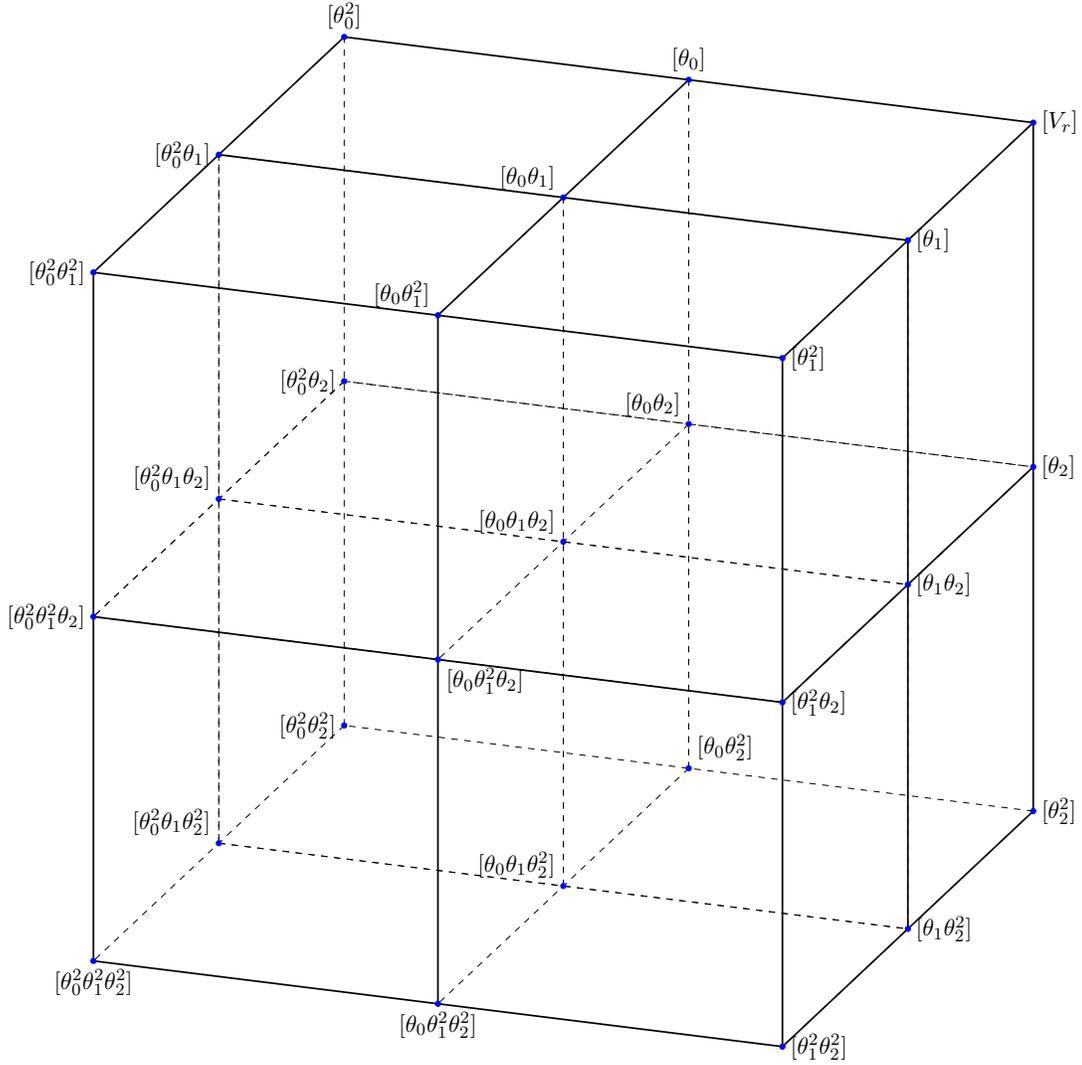
\begin{center}\resizebox{.9\textwidth}{!}{
\tdplotsetmaincoords{70}{110}
\begin{tikzpicture}[tdplot_main_coords, scale=3, xshift=-6cm]
% Vertices
\coordinate (A0) at (0,0,0);
\coordinate (A1) at (4,0,0);
\coordinate (A2) at (4,4,0);
\coordinate (A3) at (0,4,0);
\coordinate (A4) at (0,0,4);
\coordinate (A5) at (4,0,4);
\coordinate (A6) at (4,4,4);
\coordinate (A7) at (0,4,4);

% Draw cube edges
%\draw[thick] (A0)--(A1)--(A2)--(A3)--cycle;
\draw[dashed](A0)--(A1);
\draw[dashed](A3)--(A0);
\draw[thick](A1)--(A2);
\draw[thick](A2)--(A3);
\draw[dashed](A0)--(A4);
\draw[thick](A4)--(A5);
\draw[thick](A5)--(A1);
%\draw[thick] (A0)--(A4)--(A5)--(A1);

\draw[thick] (A4)--(A7)--(A6)--(A5);

\draw[thick] (A7)--(A3);
\draw[thick] (A2)--(A6);

% Edge midpoints
\coordinate (M01) at ($(A0)!0.5!(A1)$);
\coordinate (M12) at ($(A1)!0.5!(A2)$);
\coordinate (M23) at ($(A2)!0.5!(A3)$);
\coordinate (M30) at ($(A3)!0.5!(A0)$);
\coordinate (M45) at ($(A4)!0.5!(A5)$);
\coordinate (M56) at ($(A5)!0.5!(A6)$);
\coordinate (M67) at ($(A6)!0.5!(A7)$);
\coordinate (M74) at ($(A7)!0.5!(A4)$);
\coordinate (M04) at ($(A0)!0.5!(A4)$);
\coordinate (M15) at ($(A1)!0.5!(A5)$);
\coordinate (M26) at ($(A2)!0.5!(A6)$);
\coordinate (M37) at ($(A3)!0.5!(A7)$);

% Face centers
\coordinate (FBot) at (2,2,0);
\coordinate (FTop) at (2,2,4);
\coordinate (FLeft) at (0,2,2);
\coordinate (FRight) at (4,2,2);
\coordinate (FFront) at (2,0,2);
\coordinate (FBack) at (2,4,2);

% Cube center
\coordinate (C) at (2,2,2);

% Dashed lines to divide each face
% Bottom face
\draw[dashed] (2,0,0) -- (2,4,0);
\draw[dashed] (0,2,0) -- (4,2,0);

% Top face
\draw[thick] (2,0,4) -- (2,4,4);
\draw[thick] (0,2,4) -- (4,2,4);

% Left face
\draw[dashed] (0,2,0) -- (0,2,4);
\draw[dashed] (0,0,2) -- (0,4,2);

% Right face
\draw[thick] (4,2,0) -- (4,2,4);
\draw[thick] (4,0,2) -- (4,4,2);

% Front face
\draw[dashed] (2,0,0) -- (2,0,4);
\draw[dashed] (0,0,2) -- (4,0,2);

% Back face
\draw[thick] (2,4,0) -- (2,4,4);
\draw[thick] (0,4,2) -- (4,4,2);

% Middle sheet at z = 2
% Square from (0,0,2) to (4,4,2), divided into 4 parts
\draw[dashed] (0,0,2) -- (4,0,2) -- (4,4,2) -- (0,4,2) -- cycle; % outline
\draw[dashed] (2,0,2) -- (2,4,2); % vertical divider
\draw[dashed] (0,2,2) -- (4,2,2); % horizontal divider

% Middle standing sheet at x = 2
% Square from (2,0,0) to (2,4,4), divided into 4 parts
\draw[dashed] (2,0,0) -- (2,4,0) -- (2,4,4) -- (2,0,4) -- cycle; % outline
\draw[dashed] (2,2,0) -- (2,2,4); % vertical divider
\draw[dashed] (2,0,2) -- (2,4,2); % horizontal divider

% === Label all key points ===

\foreach \point/\name/\pos in {
A0/$[\theta_0^2\theta_2^2]$/left,
A1/$[\theta_0^2\theta_1^2\theta_2^2]$/below,
A2/$[\theta_1^2\theta_2^2]$/right,
  A3/$[\theta_2^2]$/right,
  A4/$[\theta_0^2]$/above,
A5/$[\theta_0^2\theta_1^2]$/left,
A6/$[\theta_1^2]$/right,
A7/$[V_r]$/right,
M01/$[\theta_0^2\theta_1\theta_2^2]$/above left,
M12/$[\theta_0\theta_1^2\theta_2^2]$/below,
M23/$[\theta_1\theta_2^2]$/right,
M30/$[\theta_0\theta_2^2]$/above right,
M45/$[\theta_0^2\theta_1]$/left,
M56/$[\theta_0\theta_1^2]$/above left,
M67/$[\theta_1]$/right,
M74/$[\theta_0]$/above,
M04/$[\theta_0^2\theta_2]$/left,
M15/$[\theta_0^2\theta_1^2\theta_2]$/left,
M26/$[\theta_1^2\theta_2]$/right,
M37/$[\theta_2]$/right,
FBot/$[\theta_0\theta_1\theta_2^2]$/above left,
FTop/$[\theta_0\theta_1]$/above left,
FLeft/$[\theta_0\theta_2]$/above left,
FRight/$[\theta_0\theta_1^2\theta_2]$/below right,
FFront/$[\theta_0^2\theta_1\theta_2]$/above left,
FBack/$[\theta_1\theta_2]$/right,
C/$[\theta_0\theta_1\theta_2]$/above left}{ \filldraw[blue] (\point) circle (0.4pt);  \node[\pos] at (\point) {\name}; }\end{tikzpicture}}
\end{center}
  
\noindent As the inductive step requires one to prove that $V_r^{(m)}/V_r^{(m+1)}$ is decomposable into principal series, it is sufficient to study the theta filtration on $V_r^{(m)}/V_r^{(m+1)}.$ 

For instance, the theta filtration on $V_r^{(m)}/V_r^{(m+1)}$ for $m=2$ and $f=3$ can be expressed by the following diagram. This diagram is obtained from the 
previous one by taking paths along the above diagram 
(starting from $[\theta_0^2]$, $[\theta_1^2]$ and $[\theta_2^2]$ and heading in a positive direction towards $[\theta_0^2\theta_1^2\theta_2^2]$).

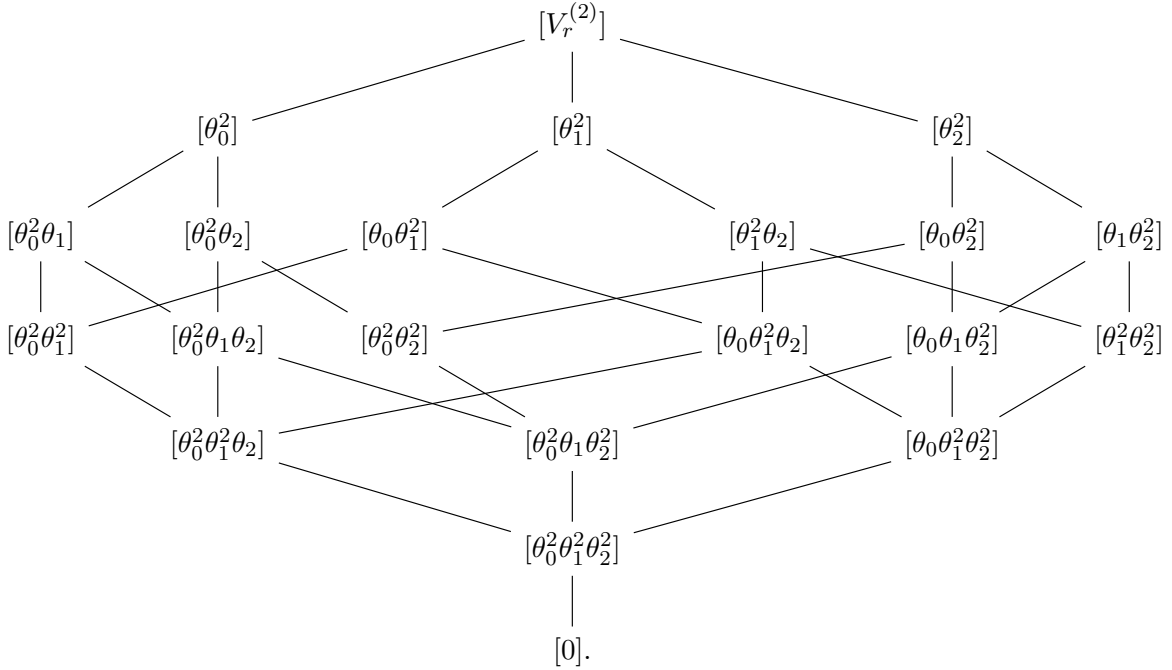
\[\begin{tikzcd}
	&&& {[V_r^{(2)}]} \\
	& {[\theta_0^2]} && {[\theta_1^2]} && {[\theta_2^2]} \\
	{[\theta_0^2\theta_1]} & {[\theta_0^2\theta_2]} & {[\theta_0\theta_1^2]} && {[\theta_1^2\theta_2]} & {[\theta_0\theta_2^2]} & {[\theta_1\theta_2^2]} \\
	{[\theta_0^2\theta_1^2]} & {[\theta_0^2\theta_1\theta_2]} & {[\theta_0^2\theta_2^2]} && {[\theta_0\theta_1^2\theta_2]} & {[\theta_0\theta_1\theta_2^2]} & {[\theta_1^2\theta_2^2]} \\
	& {[\theta_0^2\theta_1^2\theta_2]} && {[\theta_0^2\theta_1\theta_2^2]} && {[\theta_0\theta_1^2\theta_2^2]} \\
	&&& {[\theta_0^2\theta_1^2\theta_2^2]} \\
	&&& {[0].}
	\arrow[no head, from=1-4, to=2-2]
	\arrow[no head, from=1-4, to=2-4]
	\arrow[no head, from=1-4, to=2-6]
	\arrow[no head, from=2-2, to=3-1]
	\arrow[no head, from=2-2, to=3-2]
	\arrow[no head, from=2-4, to=3-3]
	\arrow[no head, from=2-4, to=3-5]
	\arrow[no head, from=2-6, to=3-6]
	\arrow[no head, from=2-6, to=3-7]
	\arrow[no head, from=3-1, to=4-1]
	\arrow[no head, from=3-1, to=4-2]
	\arrow[no head, from=3-2, to=4-2]
	\arrow[no head, from=3-2, to=4-3]
	\arrow[no head, from=3-3, to=4-1]
	\arrow[no head, from=3-3, to=4-5]
	\arrow[no head, from=3-5, to=4-5]
	\arrow[no head, from=3-5, to=4-7]
	\arrow[no head, from=3-6, to=4-3]
	\arrow[no head, from=3-6, to=4-6]
	\arrow[no head, from=3-7, to=4-6]
	\arrow[no head, from=3-7, to=4-7]
	\arrow[no head, from=4-1, to=5-2]
	\arrow[no head, from=4-2, to=5-2]
	\arrow[no head, from=4-2, to=5-4]
	\arrow[no head, from=4-3, to=5-4]
	\arrow[no head, from=4-5, to=5-2]
	\arrow[no head, from=4-5, to=5-6]
	\arrow[no head, from=4-6, to=5-4]
	\arrow[no head, from=4-6, to=5-6]
	\arrow[no head, from=4-7, to=5-6]
	\arrow[no head, from=5-2, to=6-4]
	\arrow[no head, from=5-4, to=6-4]
	\arrow[no head, from=5-6, to=6-4]
	\arrow[no head, from=6-4, to=7-4]
\end{tikzcd}\]

 The theta filtration  on \(V_r^{(m)}/V_r^{(m+1)}\) can be arranged in rows of modules generated by appropriate products of powers of the polynomials $\theta_i$ where the sum of all the powers in a particular row is constant and at least one power is $m$. Each successive row (after the top one) is indexed by level \(0,1,2,...\). The top row consists only of  \([V_r^{(m)}] = V_r^{(m)}/V_r^{(m+1)}\) and is assigned level \(-1\).
%and consists of product of powers of the \(\theta_i\) polynomials. 
Any general submodule in  a row of level \(n\) is given by 
\begin{equation}
    \left[\Biggl \langle \prod_{i=0}^{f-1}\theta_i^{j_i} \Biggr \rangle \right]\ \text{ with } \sum_{i=0}^{f-1}j_i=m+n;\ 0\le j_i \le m;\ \exists\ i\text{ such that }j_i=m . \label{level}
\end{equation}

For example, in the diagram above, the top object $[V_r^{(2)}]$ has level $-1$, and the submodules 
in the row of level $1$ must satisfy $j_0+j_1+j_2=2+1$ and hence $(j_0,j_1,j_2)$ is one of $(2,1,0)$, $(2,0,1)$, $(1,2,0)$, $(0,2,1)$, $(1,0,2)$, $(0,1,2).$ 

The containments are such that any submodule \([\langle f \rangle]\) in level \(n\) contains  \([\langle g\rangle]\) in 
level \(n+1\) if and only if \(f|g\) and this happens when $g=f\theta_i$ for some $i.$ %The bottom-most level contains \([0]\). Our claim is that all the resulting sub-quotients are principal series.

\begin{theorem}\label{principal}
    For $r_i \geq m + mq + q$, \(V_r/V_r^{(m+1)}\) is decomposable into principal series. 
\end{theorem}
\begin{proof}
    The proof is by induction on \(m\).
    The base case \(m=0\) follows from \cite{GhateJana+2025+1503+1543}, since \(V_r/V_r^*\) is itself a principal series.
    Assume the statement of the theorem holds for $m-1$ for some \(m \geq 1\). Consider the theta filtration of \(V_r/V_r^{(m+1)}\). 
    There is an exact sequence 
    $$0 \rightarrow V_r^{(m)}/V_r^{(m+1)} \rightarrow V_r/V_r^{(m+1)} \rightarrow V_r/V_r^{(m)} \rightarrow 0. $$
    %The top most sub-quotient is
    %\[
    %\frac{[V_r]}{\sum_{i=0}^{f-1}[\theta_i^m]} \cong %\frac{[V_r]}{[\langle\theta_0^m,...,\theta_{f-%1}^m\rangle]}.
    %\]
    By the induction hypothesis, the right most term \(V_r/V_r^{(m)}\) is decomposable into principal series. So, to complete the inductive step, it is sufficient to check that the left most term is 
    decomposable into principal series. 
    
    The bottom-most term of the leftmost term $V_r^{(m)}/V_r^{(m+1)}$ is 
    \begin{align*}
    [\langle \theta_0^m\theta_1^m\cdots\theta_{f-1}^m\rangle]&\cong\frac{\langle \theta_0^m\theta_1^m\cdots\theta_{f-1}^m\rangle}{ \langle \theta_0^{m+1},\theta_1^{m+1},\ldots,\theta_{f-1}^{m+1}\rangle\cap\langle \theta_0^m\theta_1^m\cdots\theta_{f-1}^m \rangle}.
    \end{align*}
    We claim that the denominator can be written as \begin{equation}\label{intersection}
\langle\theta_0^m\theta_1^m\cdots\theta_{f-1}^m\rangle \cap \langle\theta_0^{m+1},\theta_1^{m+1},\ldots,\theta_{f-1}^{m+1}\rangle = \Bigg\langle \bigcup_{l=0}^{f-1}
    \Bigg\{\left(\prod_{\substack{i=0\\ i\ne l}}^{f-1} \theta_i^{m}\right) \left(\theta_l^{m+1}\right)\Bigg\} \Bigg\rangle.\end{equation}
    It is clear that the right hand side is a subset of the left hand side. Now we prove the reverse containment. If an element $P$ lies in the intersection, then  we can write
    \begin{equation}
         P = Q\theta_0^m\theta_1^m\cdots\theta_{f-1}^m = A_0\theta_0^{m+1} + \cdots + A_{f-1}\theta_{f-1}^{m+1}\label{inteqn}
    \end{equation}
    where $Q,A_0,\cdots ,A_{f-1}$ are polynomials.
    We now apply the differential operator \(\nabla_j= a^{p^j}\frac{\partial}{\partial x_j} + b^{p^j}\frac{\partial}{\partial y_j}\) defined in \cite{GhateJana+2025+1503+1543} on both sides of \eqref{inteqn}. 
    %and use Lemma 2.14 from \cite{GhateJana+2025+1503+1543}. 
    Let $\alpha =\left(\begin{smallmatrix}
        a & b \\ c & d
    \end{smallmatrix}\right) \in \mathrm{GL}_2(\Fq)$ and let $\overrightarrow{(c,d)} =(c,d,c^p,d^p,\ldots,c^{p^{f-1}},d^{p^{f-1}}) $. %Let $\nabla_j 
    Then, by \cite[Lemma 2.14]{GhateJana+2025+1503+1543}, we have
    \begin{eqnarray*}
        \left(\prod_{i=0}^{f-1} \nabla_i^m\right)\!\!\!\!\!\! \!&&\!\!\!\!\!\!(Q\theta_0^m\theta_1^m\cdots\theta_{f-1}^m)\Big|_{\overrightarrow{(c,d)}}\\
        &=&\sum\limits_{k_{f-1}=0}^{m}\cdots\sum\limits_{k_0=0}^{m}\binom{m}{k_{f-1}}\cdots \binom{m}{k_0} \left( \left(\prod_{i=0}^{f-1} \nabla_i^{k_i}\right)(Q) \cdot \left(\prod_{i=0}^{f-1} \nabla_i^{m-k_i}\right) (\theta_0^m\theta_1^m\cdots\theta_{f-1}^m)\right)\Big |_{\overrightarrow{(c,d)}}\\ &=&0+ \left( Q\left(\prod_{i=0}^{f-1} \nabla_i^m\right)(\theta_0^m\theta_1^m\cdots\theta_{f-1}^m) \right)\Big|_{\overrightarrow{(c,d)}}
        \\
        &=& \left(Q  \cdot (m!)^f \prod_{i=0}^{f-1} (\nabla_i(\theta_i))^m \right)\Big|_{\overrightarrow{(c,d)}} \\
        &=& Q(\overrightarrow{(c,d)})(m!)^f\prod_{i=0}^{f-1} \mathrm{det} (\alpha)^{mp^i}. 
        \end{eqnarray*}
        Similarly, we have
        \begin{eqnarray*}
    \left(\prod_{i=0}^{f-1} \nabla_i^m\right)\!\!\!\!\!\!\! &&\!\!\!\!\!\!\left( \sum_j A_j \theta_j^{m+1}\right)\Big|_{\overrightarrow{(c,d)}} \\\\&=&\sum\limits_{j=0}^{f-1}\sum\limits_{k_{f-1}=0}^{m}\cdots\sum\limits_{k_0=0}^{m}\binom{m}{k_{f-1}}\cdots \binom{m}{k_0}\left( \left(\prod_{i=0}^{f-1} \nabla_i^{k_i}\right)(A_j) \cdot \left(\prod_{i=0}^{f-1} \nabla_i^{m-k_i}\right) (\theta_{j}^{m+1})\right)\Big |_{\overrightarrow{(c,d)}}\\
    &=& 0.
    \end{eqnarray*}
    Since $\mathrm{det} (\alpha) \not= 0 $  this implies $Q(c,d,c^{p},d^p,...,c^{p^{f-1}},d^{p^{f-1}}) = 0.$ Since this is true for arbitrary $\alpha =\left(\begin{smallmatrix}
        a & b \\ c & d
    \end{smallmatrix} \right) \in\mathrm{GL}_2(\Fq), $ by Lemmas 2.15, 2.16, 2.17  in \cite{GhateJana+2025+1503+1543} we know $Q \in \langle\theta_0,\ldots,\theta_{f-1}\rangle$. This implies $$P = Q\theta_0^m\cdots \theta_{f-1}^m \in \\ \Bigg\langle \bigcup_{l=0}^{f-1}
    \Bigg\{\left(\prod_{\substack{i=0\\ i\ne l}}^{f-1} \theta_i^{m}\right) \left(\theta_l^{m+1}\right)\Bigg\} \Bigg\rangle. $$
    Thus we have,
    \[[\langle \theta_0^m\theta_1^m\cdots\theta_{f-1}^m\rangle]=\frac{\langle \theta_0^m\theta_1^m\cdots\theta_{f-1}^m\rangle}{ \Big\langle \cup_{l=0}^{f-1}
    \Big\{\left(\prod_{\substack{i=0\\ i\ne l}}^{f-1} \theta_i^{m}\right) \left(\theta_l^{m+1}\right)\Big\} \Big\rangle}.\]
    Now, we use Proposition \ref{iso1} to conclude that the right hand side is a principal series.

Now we shall define a filtration of submodules on
$V_r^{(m)}/V_r^{(m+1)}$ for which every sub-quotient 
is a principal series. It would help to keep the
diagram above \eqref{level} in mind while reading
the discussion below. 
Enumerate the generators $P_0, P_1,P_2,...$ of the submodules in the filtration \eqref{level} as 
follows. Start with $P_0 = 0$. Then take $P_1 = \theta_0^{m}\cdots\theta_{f-1}^m$ from the second row from the bottom. Then move up by one row, and enumerate the generators of the modules from left to right. Repeat this process for each higher row. Let $$M_j := [\langle P_0,P_1,...,P_j\rangle].$$ 
The $M_j$ define an increasing sequence of submodules (an exhaustive increasing filtration) of  $V_r^{(m)}/V_r^{(m+1)}$.
%In this way we obtain a filtration of submodules of $V_r^{(m)}/V_r^{(m+1)}$, where any submodule in this filtration is of the form 

We show that any sub-quotient in this filtration is a principal series. We have already just shown
that $M_1/M_0$ is a principal series. 
Suppose that $P_j=\prod_{i=0}^{f-1}\theta_i^{j_i}.$ 
By Lemma \ref{iso3}, we have 
  \[ \frac{M_j}{M_{j-1}} =\frac{[\langle P_0,..., P_j\rangle]}{[ \langle P_0,...,P_{j-1} \rangle]} \cong \frac{\langle P_j\rangle}{(\langle P_0,..., P_{j-1}\rangle+ V_r^{(m+1)})\cap \langle P_j \rangle}.\]
We claim that $$(\langle P_{0},..., P_{j-1}\rangle+ V_r^{(m+1)})\cap \langle P_j \rangle=\Bigg\langle\bigcup_{i=0}^{f-1}\left\{\prod_{\substack{l=0\\l\ne i}}^{f-1}\left(\theta_l^{j_l} \right)\left(\theta_i^{j_i+1} \right) \right\}\Bigg\rangle.$$
It is easy to see that the right hand side is contained in the left hand side. Indeed, if $j_i = m$, then 
$\prod_{\substack{l=0\\l\ne i}}^{f-1}\left(\theta_l^{j_l} \right)\left(\theta_i^{j_i+1} \right) \in V_r^{(m+1)}$, and if $j_i < m$ it lies
in $\langle P_0,..., P_{j-1} \rangle$ since it is
in a row below the row in which $P_j$ lies.
The proof of the fact that the left hand side is contained in the right hand side is similar to the proof that the intersection on the left hand side of \eqref{intersection} is contained in the right hand side of \eqref{intersection}. Thus
\[
    \frac{M_j}{M_{j-1}}\cong\frac{\Bigg\langle \prod_{l=0}^{f-1}\theta_l^{j_l} \Bigg \rangle}{\Bigg \langle \bigcup_{i=0}^{f-1}\left\{
    \left(\prod_{\substack{l=0\\l\ne i}}^{f-1}\theta_l^{j_l}\right)(\theta_i^{j_i+1})
    \right\}
    \Bigg \rangle}
\]
which, by Proposition \ref{iso1}, is a principal series.
It follows that \(V_r^{(m)}/V_r^{(m+1)}\) is decomposable into principal series. 

We conclude that \(V_r/V_r^{(m+1)}\) is decomposable into principal series by induction.
\end{proof}
\begin{comment}
\begin{proof}
    Take any enumeration of the diagram, $\langle P_0\rangle,\langle P_1 \rangle,\cdots$ such that any element of row $n$ appears after every element of row $n+1$. We claim that $\frac{[\langle P_0 \rangle + \cdots + \langle P_i\rangle]}{[ \langle P_0 \rangle+ \cdots + \langle P_{i-1} \rangle]}$ is a principal series. By the isomorphism theorems we have 
    \[\frac{[\langle P_0 \rangle + \cdots + \langle P_i\rangle]}{[ \langle P_0 \rangle+ \cdots + \langle P_{i-1} \rangle]} \cong \frac{\langle P_i\rangle}{\langle P_0 \rangle\cap \langle P_i \rangle + \cdots \langle P_{i-1}\rangle \cap \langle P_i \rangle + V_r^{(m+1)}\cap \langle P_i \rangle}\]

    Now by \cref{intersection_lemma} we are quotienting by the sum of subspaces given by $\mathrm{lcm}(P_i,P_j)$. Notice that the lcm of any two elements in the diagram lies in the diagram at a row strictly lower than both. For this reason we only need to consider the terms given by polynomials in row $n+1$ that $P_i$ divides. By \cref{iso1} the quotient is a principal series.
\end{proof}
\end{comment}

Since the number of steps in the filtration above
is equal to the number of non-zero generators which in turn is obtained by deleting the hypercube with
$m$ subdivisions on each side from the one with
$m+1$ divisions on each side, we in fact obtain
the sharper result:

\begin{corollary}
  \label{number}
Let $m \geq 0$. Then
  \begin{itemize}
  \item \(V_r^{(m)}/V_r^{(m+1)}\) is decomposable into $(m+1)^f - m^f$ principal series.
  \item  \(V_r/V_r^{(m+1)}\) is decomposable into $(m+1)^f$ principal series. 
  \end{itemize}
\end{corollary}
\section{Analogy between \(V_r/V_r^*\) and \(V_r/V_r^{**}\)}

We know that \(V_r/V_r^*\) is a principal series \cite{sandra}, \cite{GhateJana+2025+1503+1543}, and 
Breuil's notes  \cite{Breuil} describes the explicit structure of the (generically) $2^f$ irreducible representations in its socle filtration. By the previous section, \(V_r/V_r^{**}\) can be decomposed into $2^f$ principal series (see Corollary~\ref{number}). We may define a graph by connecting two of these principal series by an edge if they occur in an extension (of a natural kind that we shall describe below). 
In this section, we observe that both the structures are identical to a %directed analogue of the 
directed hypercube graph.

First, we recall the definition of the directed hypercube graph \(\bar{Q}_n\) on \(n\) vertices. The vertices are given by \(V=P(\{1,2,...,n\})\). Here \(P(S)\) denotes the power set of the set \(S\). There is an edge from a vertex \(u\) to a vertex \(v\) if \(u\subset v\) and \(|v|=|u|+1\), where \(|X|\) denotes the cardinality of the set \(X\).

\subsection{\(V_r/V_r^{*}\)}
First we examine the extensions between the irreducible sub-quotients of \(V_r/V_r^{*}\). In Breuil's notes \cite{Breuil}, Theorem 7.6 describes the extensions between the irreducible sub-quotients of a generic principal series. Before stating the theorem, we introduce some notation stated in the above reference. 

Let \(\mathcal{P}(x_0,...,x_{f-1})\) be the set of \(f\)-tuples \(\lambda=(\lambda_0(x_0),...,\lambda_{f-1}(x_{f-1}))\) defined as follows. If \(f=1\), \(\lambda_0(x_0)\in\{x_0,p-1-x_0\}\). If \(f>1\) then
\begin{enumerate}
    \item \(\lambda_i(x_i)\in\{x_i,x_i-1,p-2-x_i,p-1-x_i\}\) for all \(i\).
    \item If \(\lambda_i(x_i)\in\{x_i,\ x_{i}-1\}\), then \(\lambda_{i+1}(x_{i+1}) \in \{ x_{i+1},\ p-2-x_{i+1}\}\).
    \item If \(\lambda_i(x_i)\in\{p-2-x_i,\ p-1-x_{i}\}\), then \(\lambda_{i+1}(x_{i+1}) \in \{ p-1-x_{i+1},\ x_{i+1}-1\}\).
\end{enumerate}
    We adopt the conventions that \(x_f=x_0\) and \(\lambda_f(x_f)=\lambda_0(x_0)\).

    For \(\lambda \in \mathcal{P}(x_0,...,x_{f-1})\), define
    \[
    \mathcal{S}(\lambda):= \{i\in\{0,1,...,f-1\} \text{ such that } \lambda_i\in\{p-1-x_i,x_i-1\}\},
    \]
    \[
    l(\lambda):=|\mathcal{S}(\lambda)|;\text{ write }\lambda\le\lambda'\text{ if }\mathcal{S}(\lambda)\subset \mathcal{S}(\lambda').
    \]
We recall the part of Theorem 7.6 in \cite{Breuil}, which will be relevant in this work.
\begin{theorem}Let \(\chi: \begin{pmatrix}a &b\\0& d \end{pmatrix} \mapsto d^r,\ r\notin \{0,\ q-1\}\).
\begin{enumerate}\label{breuil}
    \item The irreducible sub-quotients of \(\mathrm{ind}_{B(\mathbb{F}_q)}^{G(\mathbb{F}_q)} \chi \) are all the distinct weights (twisted by some power of the determinant $D$):
    \[(\lambda_0(r_0),...,\lambda_{f-1}(r_{f-1}))\]
    for \(\lambda\in\mathcal{P}(x_0,...,x_{f-1})\), forgetting the weights such that \(\lambda_i<0\) for some \(i\).
    \item If \(\tau,\ \tau'\)are irreducible sub-quotients of \(\mathrm{ind}_{B(\mathbb{F}_q)}^{G(\mathbb{F}_q)} \chi \), we write \(\tau'\le\tau\) if the corresponding \(f\)-tuples \(\lambda',\ \lambda\) in (1) satisfy \(\lambda'\le \lambda\). Let \(\tau\) be an irreducible sub-quotient of \(\mathrm{ind}_{B(\mathbb{F}_q)}^{G(\mathbb{F}_q)}\chi\) and \(Q(\tau)\) the unique quotient with socle \(\tau\). Then the socle and co-socle filtrations on \(Q(\tau)\) are the same (up to renumbering), with graded pieces:
    \[
    (Q(\tau))_i=\bigoplus_{\substack{l(\tau')=i+l(\tau)\\ \tau\le \tau'}}\tau'
    \]
    for \(0\le i\le f-l(\tau)\).
\end{enumerate}
    
\end{theorem}

Since we know the set of irreducible sub-quotients by the above theorem, and we want to find a bijection with \(P(\{1,2,...,f\})\), we prove the following lemma.
    \begin{lemma}\label{bijection}
        Given \(X\subset \{1,2,...,f\}\), there exists a unique \(\lambda_{X} \in \mathcal{P}(x_0,...,x_{f-1})\), such that \(\mathcal{S}(\lambda_{X})=X.\) 
    \end{lemma}
    \begin{proof}
        For this proof, if \(\lambda_i(x_i)\in\{p-1-x_i,\ x_i-1\}\), we will say that \(\lambda_i\) has parity \(1\), else we say that \(\lambda_i\) has parity \(0\). First, we prove that for any \(a,\ b \in \{0,1\}\), there exists unique choice of \(\lambda_i(x_i)\) such that parity of \(\lambda_i(x_i)=a\) and possible parity of \(\lambda_{i+1}(x_{i+1})=b\). This is true because we can check by brute force that for each of the four choices of \(a, b\), this is true. For example, if \(a=0,\) \( b=0\), then the only choice of \(\lambda_i(x_i)\) is \(x_i\).

        Now given \(X\subset \{0,1,...,f\}\), if we want \(\lambda_{X}\) such that \(\mathcal{S}(\lambda_{X})=X\), we have already fixed the parity of each of the \(\lambda_i(x_i)\)'s. Hence the value of the \(\lambda_i(x_i)\)'s are also uniquely determined. Thus, the entire tuple \(\lambda_{X}\) is uniquely defined. Hence proved.
    \end{proof}
We conclude this sub-section with the following proposition. Note that the claim is not true for all \(V_r\).

\begin{proposition} Write \(r\equiv\sum_{i=0}^{f-1}a_ip^i\) mod $(q-1)$, with \(0\le a_i<p\) for all \(i\). Assume \(a_i\notin \{0,p-1\} \) for any \(i\). Let \(V\) be the set of irreducible sub-quotients of \(V_r/V_r^{*}\). Let \(E=\{(u,v)\in V^2|\ u\text{ has an extension over v} \}\). Then, the directed graph \(G(V,E)\) is isomorphic to \(\bar{Q}_f\).
\end{proposition}
\begin{proof}
    By \cite{GhateJana+2025+1503+1543}, we know that \(V_r/V_r^{*}\cong \text{ind}_{B(\mathbb{F}_q)}^{G(\mathbb{F}_q)}d^r\). Notice that if \(a_i\notin \{0,p-1\} \), for every possible value of \(\lambda_i(x_i)\), then \(\lambda_i(a_i)\ge0\), hence for every \(\lambda \in \mathcal P (x_0,...,x_{f-1})\), \((\lambda_0(a_0),...,\lambda_{f-1}(a_{f-1}))\) appears in \(\text{ind}_{B(\mathbb{F}_q)}^{G(\mathbb{F}_q)}d^r\) up to twist.

    By Lemma \ref{bijection}, the map \(X\mapsto \lambda_{X}\) is a bijection. Since each \(\lambda_{X}\) corresponds to an irreducible sub-quotient \(\tau_X\) of \(\text{ind}_{B(\mathbb{F}_q)}^{G(\mathbb{F}_q)}d^r\), this establishes a bijection between the vertices of \(\bar{Q}_f\) and \(V\) where \(X\mapsto \tau_X\). Now, we need to establish the correspondence between the edges. Substituting \(i=1\) in part (2) of \ref{breuil}, we obtain 
    \[
    (Q(\tau))_1=\bigoplus_{\substack{l(\tau')=1+l(\tau)\\ \tau\le \tau'}}\tau'.
    \]
    So, \(\tau_X\) extends over \(\tau_{X'}\) if and only if the corresponding \(\lambda_X, \ \lambda_{X'}\) satisfy
    \[
    |\mathcal{S}(\lambda_{X'})|=1+|\mathcal{S}(\lambda_X)|;\  \lambda_X\le \lambda_{X'}.
    \]
    The above happens if and only if \(X,\ X'\) satisfy 
    \[
        |X'|=|X|+1;\ X\subset X'.
    \]
    The above conditions are identical to the conditions for the existence of an edge from \(X\) to \(X'\) in \(\bar{Q}_f\). So there is an edge from \(\tau_{X}\) to \(\tau_{X'}\) in \(G(V,E)\) if and only if there is an edge from \(X\) to \(X'\) in \(\bar{Q}_f\). Hence the map \(X\mapsto \tau_X\) is a graph isomorphism from \(\bar{Q}_f\) to \(G(V,E)\).
\end{proof}

\subsection{\(V_r/V_r^{**}\)}

Now, we show that the extensions between certain `adjacent' principal series representations present in \(V_r/V_r^{**}\) as sub-quotients form a directed hypercube graph.

\begin{proposition}
    Let \(V\) be the set of principal series representation 
     present as a sub-quotient in  \(V_r/V_r^{**}\) of the
     form \eqref{VX} below. Let \(E=\{(u,v)\in V^2|\ u\) has an extension by \(v\) of the form \eqref{extension} below\(\}\). Then, the directed graph \(G(V,E)\) is isomorphic to \(\bar{Q}_f\).
\end{proposition}

\begin{proof}
    Consider the theta
    filtration on $V_r/V_r^{**}$ from the previous section. 
    %which gives the principal series representations as sub-quotients. 
    The level \(n\) part of the 
    filtration of \(V_r/V_r^{**}\) is
\[\left\{\left[\Biggl \langle \prod_{i=0}^{f-1}\theta_i^{j_i} \Biggr \rangle \right]\ \Bigg| \sum_{i=0}^{f-1}j_i=n+1;\ 0\le j_i \le 1%;\ \exists\ i\text{ such that }j_i=1 
\right\}\]
for $-1 \leq  n \leq f-1$.
Since all the \(j_i\)'s are \(0\) or \(1\), the above set can also be written as 
    \[\left\{\left[\Biggl \langle \prod_{i\in X}\theta_i \Biggr \rangle \right]\ \Bigg|\ X \subset \{1,2,...,f\},\ |X|=n+1 \right\}.\]

For any \(X \subset \{1,2,...,f\}\), let \(P_X\) denote the polynomial \(\prod_{i\in X}\theta_i\). If \(X=\phi\), then \(P_X\) is the constant polynomial \(1\). Let \(V_X\) denote the principal series sub-quotient of \(V_r/V_r^{**}\) defined by 
\begin{eqnarray}
      \label{VX}
V_X = \dfrac{[\langle P_X \rangle]}{[\langle P_1, P_2, \ldots, P_t \rangle]} \cong \frac{\langle P_X \rangle}{(\langle P_1, P_2,\ldots,P_{t} \rangle + V_r^{(m+1)})\cap\langle P_X \rangle},
\end{eqnarray}
where $P_1, P_2, \ldots, P_t$ are the $P_Y$ with $X \subset Y$ and $|Y|=|X|+1$. The fact that $V_X$ is indeed a principal series follows since simplifying the denominator of the rightmost subquotient in (\ref{VX}) shows that \(V_X\) is of the form \(M_j/M_{j-1}\), which was defined and proved to be a principal series in the proof of Theorem \ref{principal}. Given such $X$ and $Y$ with $P_Y$ equal to say $P_1$,
there is an extension 
\begin{eqnarray}
    \label{extension}
  0\rightarrow V_Y \cong \dfrac{[\langle P_Y, P_2, \ldots, P_t \rangle]}{[\langle P_2, \ldots, P_t, Q_1, Q_2, \ldots, Q_k \rangle]} \rightarrow E_{X,Y} = 
\dfrac{[\langle P_X \rangle]}{[\langle P_2, \ldots, P_t, Q_1, Q_2, \ldots, Q_k \rangle]}
\rightarrow V_X \rightarrow 0
\end{eqnarray}
where $Q_1, Q_2, \ldots, Q_k$ are the $P_{Y'}$ for the
sets $Y \subset Y'$ with  $|Y'|=|Y|+1$.

%Now, every \(X\subset\{1,2,...,f\}\) gives a principal series representation, \(V_X\), and every principal series representations which are sub-quotients of \(V_r/V_r^{**}\) is \(V_X\) for some \(X\subset \{1,2,...,f\}\). Hence the \(V\) defined in the proposition is \(V=\{V_X|\ X\subset\{1,2,...,f\}\}\).

Clearly the map \(X\mapsto V_X\) is a graph isomorphism from \(\bar{Q}_f\) to \(G(V,E)\). 
%The vertices are in bijection by the map, hence we only need to check the correspondence between edges. Consider \(G(V,E)\). For \(V_X,V_Y\in V\), \((V_{X},V_Y)\in E\) if and only if the sub-module \([\langle P_X \rangle]\) extends over \([\langle P_Y \rangle]\). By the conditions written after (\ref{level}), this is equivalent to \(P_X|P_Y\) and \(|Y|=|X|+1\). Lastly we note that  \(P_X|P_Y\) if and only if \(X\subset Y\). Thus \((V_X,V_Y)\in E\) if and only if \((X,Y)\) is an edge in \(\bar{Q}_f\). Hence proved.
\end{proof}

\vspace{.2cm}

{\bf \noindent Acknowledgments:} This paper was written 
for the proceedings of the International Conference RAMRA held at SRM University in January 2025. EG
thanks the organizers K. Chakraborty and K. Banerjee for 
the invitation. This paper grew out of a question asked by 
SP while a postdoctoral fellow at TIFR. It was answered jointly with SB and SV who were VSRP students at TIFR in 
2025. 

\bibliographystyle{abbrv}

% This command includes the bibliography. The 'references.bib' file must be in the same directory.
\bibliography{ref}

\begin{thebibliography}{1}

\bibitem{doi:10.1142/S0219498826502701}
S.~Bhattacharya and A.~Ganguli.
\newblock Weights for mod $p$ quaternionic forms in the unramified case.
\newblock {\em \textnormal{To appear in }J. Algebra Appl.}, 2025.

\bibitem{Breuil}
C.~Breuil.
\newblock \textit{Representations of Galois and of $\mathrm{GL}_2$ in characteristic $p$}.
\newblock Lecture notes of a graduate course at Columbia University (Fall 2007).

\bibitem{BP}
C.~Breuil and V.~Pa{\v{s}}k{\={u}}nas.
\newblock Towards a modulo {$p$} {L}anglands correspondence for {${\rm GL}_2$}.
\newblock {\em Mem. Amer. Math. Soc.}, 216(1016):vi+114, 2012.

\bibitem{diamond}
F.~Diamond.
\newblock A correspondence between representations of local {G}alois groups and {L}ie-type groups.
\newblock In {\em {$L$}-functions and {G}alois representations}, volume 320 of {\em London Math. Soc. Lecture Note Ser.}, pages 187--206. Cambridge Univ. Press, Cambridge, 2007.

\bibitem{GhateJana+2025+1503+1543}
E.~Ghate and A.~Jana.
\newblock Modular representations of $\mathrm{GL}_2(\mathbb{F}_q)$ using calculus.
\newblock {\em Forum Mathematicum}, 37(5):1503--1543, 2025.

\bibitem{glover}
D.~J. Glover.
\newblock A study of certain modular representations.
\newblock {\em J. Algebra}, 51(2):425--475, 1978.

\bibitem{F.M.}
F.~M. Kouwenhoven.
\newblock Indecomposable representations of {$M(2,\bold F_q)$} over {$\bold F_q$}.
\newblock {\em J. Algebra}, 155(2):369--396, 1993.

\bibitem{sandra}
S.~Rozensztajn.
\newblock Asymptotic values of modular multiplicities for {$\rm GL_2$}.
\newblock {\em J. Th\'eor. Nombres Bordeaux}, 26(2):465--482, 2014.

\end{thebibliography}

%%%%%%%%%%%%%%%%%%%%%%%%%%%%%%%%%%%%%%%%%%%%%%%%%%%%%%%%%%%%%%%%%%%%%%%%%%%%%%%%
%   APPENDIX (OPTIONAL)
%%%%%%%%%%%%%%%%%%%%%%%%%%%%%%%%%%%%%%%%%%%%%%%%%%%%%%%%%%%%%%%%%%%%%%%%%%%%%%%%

\end{document}